\documentclass[pdflatex,sn-mathphys-num]{sn-jnl}

\usepackage{graphicx}%
\usepackage{multirow}%
\usepackage{amsmath,amssymb,amsfonts}%
\DeclareMathOperator*{\argmax}{argmax}

\usepackage{amsthm}%
\usepackage{mathrsfs}%
\usepackage[title]{appendix}%
\usepackage{xcolor}%
\usepackage{textcomp}%
\usepackage{manyfoot}%
\usepackage{booktabs}%
\usepackage{algorithm}%
\usepackage{algorithmicx}%
\usepackage{algpseudocode}%
\usepackage{listings}%
\usepackage{tabularx}
\usepackage{lastpage, multicol, ytableau}
\usepackage{arydshln}
\usepackage{amsthm}
\usepackage{tikz} 
\usepackage{makecell}
\usepackage{hyperref}

\theoremstyle{thmstyleone}%
\newtheorem{theorem}{Theorem}%  meant for continuous numbers
\newtheorem{proposition}{Proposition}% to get separate numbers for theorem and proposition etc.
\newtheorem{lemma}{Lemma}
\newtheorem{corollary}{Corollary}

\theoremstyle{thmstyletwo}%
\newtheorem{example}{Example}%

\theoremstyle{thmstylethree}%
\newtheorem{definition}{Definition}%

\begin{document}

\title[Sequential Apportionments from Stationary Divisor Methods]{Sequential Apportionments from Stationary Divisor Methods}

\author*[1]{\fnm{Michael A.} \sur{Jones}}\email{maj@ams.org}
\equalcont{These authors contributed equally to this work.}

\author[2]{\fnm{Brittany} \sur{Ohlinger}}\email{sheltonbc@gmail.com}
\equalcont{These authors contributed equally to this work.}

\author[3]{\fnm{Jennifer} \sur{Wilson}}\email{wilsonj@newschool.edu}
\equalcont{These authors contributed equally to this work.}

\affil*[1]{\orgdiv{Mathematical Reviews}, \orgname{American Mathematical Society}, \orgaddress{\street{ 535 W. William St., Ste. 210}, \city{Ann Arbor}, \state{Michigan}, \postcode{48103}, \country{US}}}

\affil[2]{\orgaddress{\street{102 Golf Circle}, \city{Bernville}, \state{Pennsylvania}, \postcode{19506},  \country{US}}}

\affil[3]{\orgdiv{Department of Natural Sciences and Mathematics}, \orgname{Eugene Lang College, The New School}, \orgaddress{\street{65 West 11th Street}, \city{New York}, \state{New York}, \postcode{10011}, \country{US}}}

%%==================================%%
%% Sample for unstructured abstract %%
%%==================================%%

\abstract{Divisor methods are well known to satisfy house monotonicity, which allows representative seats to be allocated sequentially. We focus on stationary divisor methods defined by a rounding cutpoint $c \in [0,1]$. For such methods with integer-valued votes, the resulting apportionment sequences are periodic. Restricting attention to two-party allocations, we characterize the set of possible sequences and establish a connection between the lexicographical ordering of these sequences and the parameter $c$. We then show how sequences for all pairs of parties can be systematically extended to the $n$-party setting.  Further,  we determine the number of distinct sequences in the $n$-party problem for all $c$.  Our approach offers a refined perspective on size bias: rather than viewing large parties as simply receiving more seats, we show that they instead obtain their seats earlier in the apportionment sequence. Of particular interest is a new relationship we uncover between the sequences generated by the smallest divisor (Adams) and  greatest divisor (D'Hondt  or Jefferson) methods.}


%%\pacs[JEL Classification]{D8, H51}


\maketitle

\section{Introduction}

Divisor methods of apportionment are widely used in proportional representation systems to allocate seats to political parties in proportion to their share of the popular vote. Less attention, however, has been given to how these methods determine the sequence in which seats are awarded. This sequence is important in practice as it is used by some coalition governments to determine cabinet positions: while the office of prime minister typically goes to the party with a plurality of votes, other cabinet positions are selected by parties based on the sequence.  For example, in Northern Ireland, the greatest divisor method (also known as D'Hondt's or Jefferson's method) is used not only to determine the number of seats allocated to each party in the Assembly but also the order in which cabinet posts are selected \cite{OGE}. Sequential apportionment in this context offers an alternative to the intense negotiation and strategic bargaining that often arise during the establishment of a new cabinet.  Using an apportionment method to allocate positions sequentially can help mitigate these conflicts and avoid delays.

Just-in-Time (JIT) sequencing problems have the same structure as the sequential allocation of minister positions.  In these problems, a number of  items must be processed (or produced) to meet a set of demands or  users. The items must be processed one at a time, and so the question arises as to what item order satisfies the demands in an equitable way.  Although divisor methods have been used to allocate the processing of items, other algorithms have been developed to solve such problems under different optimality conditions, such as minimizing the total deviation or the maximum deviation.  Kubiak  \cite{K} provides a comprehensive analysis of JIT sequencing problems, exploring the usefulness of divisor methods for JIT sequencing and analyzing the complexity of algorithms used to solve JIT sequencing problems.

Our work has application to sequential apportionment in both fields, as we focus our attention on the possible apportionment sequences from stationary divisor methods, a family parametrized by a rounding cutpoint $c \in [0, 1]$ which includes the smallest divisor method (also known as Adams' method for which $c=0$), Webster's or Sainte Lagu\"e's method ($c=0.5$), and D'Hondt's method (used in Northern Ireland, for which $c=1$). Such sequences arise  whenever apportionment methods satisfy the property of {\it house monotonicity}, which requires that no party's allocation decreases as the number of total available seats increases.  Monotonic apportionment methods give rise naturally to a sequence by considering the allocation of one seat (the first seat), the allocation of the second seat, and so on. Stationary divisor methods have an easily described algorithm for determining this sequence---versions of which are used in awarding minister positions and in JIT sequencing problems.

Stationary divisor methods also have a nice property in that the sequences they give rise to are periodic (when vote totals are integers), as proved by Kubiak \cite{K} in the context of JIT sequencing. Not all divisor methods yield periodic sequences, as evidenced by Hill-Huntington's method, which is used in the U.S. to determine the numbers of seats each state receives in the U.S. House of Representatives.  Periodicity implies that there are only a finite number of possible sequences from stationary methods. In this paper, we identify all possible sequences generated by stationary divisor methods for given party vote totals and provide explicit formulas for the total number of such sequences as functions of party vote shares. This is in contrast to work by Cembrano et al.  \cite{CCSTV}, in which they use a geometric approach to obtain bounds on the number of different possible apportionments (not sequential apportionments) for $n$ parties for the class of stationary divisor methods, as well as for power-mean divisor methods.  We also establish necessary and sufficient conditions for determining when a periodic sequence corresponds to a sequential apportionment, as well as methods for identifying the associated stationary divisor method.

We combine the periodicity of the apportionment sequences and a notion of weak proportionality for divisor methods to refine the notion of size bias of divisor methods. As identified in Balinski and Young \cite{BY}, Adams' method provides  one extreme  that favors small parties and the D'Hondt method  provides the other extreme that favors large parties.  However, by applying a lexicographic ordering of apportionment sequences, we reframe this distinction as less about an absolute bias and more about variation in when large and small parties receive their seats. The refinement also  leads to a new relationship between both the sequences and apportionments from Adams' and D'Hondt's methods.

Applying a sequential framework to the apportionment problem is a relatively unstudied field. In addition to the   literature analyzing political apportionment cited above, Pukelshein \cite{P} provides an in-depth study of the use of divisor methods  in systems of proportional representation, but does not focus on sequential apportionment.  A small number of studies have examined the use or potential use of divisor sequences for the distribution of ministerial positions. O'Leary, Grofman, and Elkit \cite{OGE} analyze  the systems used in the Northern Ireland Assembly,  in Danish municipal governments, and in the European parliament to determine committee chairs.  Raabe and Linhart \cite{RL} investigate  several models  incorporating party preferences and the sequential allocations arising from the D'Hondt, the Sainte Lagu\"e, and Adams' methods and compare them to actual cabinet  allocations in the German parliament. They note that adapting the structure of such sequences may be an advantage to parties given the time pressures they face in selecting cabinet portfolios after an election.  Ecker, Meyer, and M\"uller \cite{EMM}  take a similar modeling approach in analyzing party portfolios for 146 coalition governments in Western and Central Eastern Europe.  

When cabinet ministry positions are awarded sequentially, Brams and Kaplan \cite{BK} explain that it may not be rational for parties to naively select their most-preferred minister when available. Further, they show that strategic selection may result in a combination of ministry positions that is not Pareto-optimal, making all parties worse off, and how choosing earlier in the sequence may be harmful to a party.  For two parties, Brams and Kaplan \cite{BK} show how to combine sequential choices with a type of trading to eliminate these negative consequences.

The literature on sequential apportionment in JIT problems is more developed. Dhamala and Thapa \cite{DT} provide an overview of the relationship between JIT and apportionment problems, while Dhamala, Thapa and Yu \cite{DTY}  propose new mean-based divisor methods for minimizing the sum of deviations in a JIT sequencing problem. More along the approaches taken in social choice theory, J\'{o}zefowska, J\'{o}zefowsk, and  Kubiak \cite{JJK} provide an axiomatic approach to JIT sequencing problems, while Chun \cite{Chun} focuses on fair queueing and strategic interactions. Kubiak \cite{K} shows that fair queueing based on starting times results in the sequence from Adams’ method, while fair queueing based on finishing times results in the sequence from D'Hondt's method.  J\'{o}zefowska, J\'{o}zefowsk, and  Kubiak \cite{JJK09} show that no divisor method solves the  Liu-Layland periodic scheduling problem, which is similar to JIT sequencing problem except it includes times by which jobs must be finished.

In the next section, we review the basics of divisor methods and their application to award representatives in succession; this includes properties satisfied by divisor methods, but also new results about apportionment sequences and their period length.  Section \ref{prelims} also explains the partial ordering on apportionment sequences and concludes with a result tying the apportionment sequences from the divisor methods of Adams and D'Hondt.

In Section \ref{2parties}, we focus our attention on properties of 2-party apportionment sequences.  For any vote distribution, we characterize the possible 2-party apportionment sequences for all stationary divisor methods, breaking the interval of possible cutpoints $[0,1]$ into equivalence classes. The possible apportionment sequences for $n$ parties can be determined by lifting all $\binom{n}{2}$ pairs of 2-party sequences to the $n$-party sequence.  Using properties of the common divisors of the parties' vote totals, we are able to count the number of possible sequences for the $n$-party problem.  Our analyses for $n$-party sequential apportionment appears in  Section \ref{more_parties}. We offer some concluding remarks in Section \ref{endremarks}.  An appendix includes some of the more technical proofs.

\section{Sequential Apportionment}\label{prelims}

Proportional representation systems assign seats to parties in roughly the same proportion as the votes they receive.  The D'Hondt method---also called the greatest divisor method---assigns seats to each party sequentially by dividing their vote count by successive divisors $1, 2, 3, \dots$ and allocating the next seat to the party with the greatest ratio. At each step, the divisor for the party that received the previous seat is updated. The process is repeated  until all seats have been allocated. We illustrate with an example from the 2022 Northern Ireland Assembly. This election was notable in part because it was the first election in which Sinn F\'ein was awarded the largest number of seats.  Eight parties received seats, based on the number of votes received. The First Minister and deputy First Minister were nominated by the largest and second largest parties, respectively \cite{Education}. The remaining eight positions of the Executive Committee were selected sequentially using the D'Hondt method. However, instead of using the popular vote totals, the proportions were based on the number of seats received by each party in the Assembly. 

\begin{example} \label{NIreland}
The 2022 election for the Northern Ireland Assembly resulted in the following  eight parties receiving seats \cite{BBC}; the parties and their number of seats were:

\vskip 2mm
\begin{tabular}{ll}
Sinn F\' ein  (27) &  Democratic Unionist Party (25) \\
Alliance Party (17) & Ulster Unionist Party (9) \\
Social Democratic \& Labour Party (8) &  Independent (2) \\  
Traditional Unionist Voice (1) &  People Before Profit (1) \end{tabular}

\vskip 2mm

Sinn F\' ein (SF) and the Democratic Unionist Party (DUP) nominated individuals for the First Minister and the deputy First Minister.  The position of Justice Minister was awarded to the Alliance Party (AP) as part of a political compromise requiring cross-community support from both nationalists and unionists.  The remaining positions were then allocated to the parties sequentially.  For instance, the next position was allocated to SF because
$$ \frac{27}{1} = \max\left\{\frac{27}{1} , \frac{25}{1},  \frac{17}{2}, \frac{9}{1}, \frac{8}{1} , \frac{2}{1}, \frac{1}{1} \right\},$$
where the denominator for the Alliance Party (AP) was increased to 2 to reflect their being awarded the Justice Minister position.
The  next position was given to the DUP because
$$ \frac{25}{1} = \max\left\{ \frac{27}{2},   \frac{25}{1}, \frac{17}{2}, \frac{9}{1}, \frac{8}{1} , \frac{2}{1}, \frac{1}{1} \right\},$$
where the divisor for SF had been updated in a similar manner.

Continuing in this way, the positions were eventually allocated as follows: Departments of the Economy (SF), of Education (DUP), of Finance (SF), for Communities (DUP), of Health (UUP), of Infrastructure (SF), and of Agriculture, Environment and Rural Affairs (AP).  
This is a result of the following set of inequalities
\vskip 2mm

\begin{tabular}{ccccccccccccc}
SF & & DUP & & SF & & DUP & & UUP & & SF & & AP \\
$27/1$ & $>$ & $25/1$ & $>$ & $27/2$ & $>$ & $25/2$ & $>$ & $9/1$ & $\ge$ & $27/3$ & $>$ & $17/2$.  \end{tabular}
\vskip 2mm

Although sequential apportionment is intended to streamline negotiations over ministerial positions, the Democratic Unionist Party postponed the formation of the Executive Committee, prompting legal intervention to avert a snap election. The process was not completed until 2024, suggesting that sequential apportionment did not, in this instance, speed the formation of the Executive Committee.\end{example}

The use of the divisors  $1, 2, 3, \ldots, $  distinguishes D'Hondt's method. More generally, divisor methods can be defined through their unique set of denominators or divisors. In this article, we focus on the class of stationary divisor methods, which are parametrized by a fixed constant  $c \in [0, 1]$, and given by the sequence   $a_i+c$ where  $a_i$ is the nonnegative integer number of positions already awarded to party $i$, that is, party $i$'s current apportionment.   We refer to the value $c$ as the {\it cutpoint} of the corresponding divisor method. The process of sequential apportionment for a general cutpoint $c$ is defined below.  

Let $p_i \in \mathbb N$ be the number of votes that party $i$ receives for $i \in N = \{1, 2, \dots, n\}$ and $h \in \mathbb N$ be the number of seats to be awarded, usually referred to as the house size. Without loss of generality, we assume that parties are ordered so that  $p_i \ge p_j$ if $i > j$.  

 \begin{definition}
For fixed $0 \le c \le 1$, given a distribution of votes $\mathbf p=(p_1, p_2, \ldots, p_n)$, the apportionment sequence for the stationary divisor method with cutpoint $c$ is defined inductively as follows. Start by awarding each party  $a_i=0$ seats. Once parties $p_i$ have been awarded $a_i$ seats, award the next seat to the party that  maximizes the ratio $\frac{p_i}{a_i+c}$. Continue until all $h$ seats have been allocated.
\end{definition}

If $c=1$, then the sequential apportionment process corresponds to the D'Hondt method. If $c=1/2$, the process corresponds to the Sainte-Lagu\"{e} or Webster method. If $c=0$, the apportionment method is known as Adams' method.  Because dividing by $0$ is undefined, initial adjustments are made for Adams' method: we assume that $\frac{p_i}{0} > \frac{p_j}{0}$ whenever $p_i > p_j$ and $\frac{p_i}{0} > \frac{p_j}{k}$ for any $k > 0$. Thus, under Adams' method, the first $n$ seats are awarded, one to each party, assuming there are sufficient seats.  There are other divisor methods, such as the Hill-Huntington method and Dean's method, based on the geometric and harmonic means, respectively, that also give a seat to every party before a party gets a second seat.  (To see how Hill-Huntington's method fits into the sequential framework, it awards the next seat  to the party that maximizes the ratio $p_i/\sqrt{a_i(a_i+1)}$, where $\sqrt{a_i(a_i+1)}$ is the geometric mean of the current apportionment $a_i$ and the next apportionment $a_i + 1$.)

Because of the possibility of ties, the apportionment sequence may not be unique. We assume that ties are broken in favor of the larger parties. Therefore, if at some stage in the apportionment sequence, $p_i/(a_i+c)=p_j/(a_j+c) = \max_{k \le n}p_k/(a_k+c)$, then the seat is awarded to the party with the smallest index $i$, that is, the party with the most votes, unless there were parties with the same largest number of votes.

In practice, vote totals are very large. However, divisor methods depend only on the ratios between vote distributions. Hence, our examples will use relatively small $p_i$ values.

\begin{example}  \label{firstex}
Suppose  two parties receive votes $\mathbf p = (32, 14)$ with a house size of $h=10$.  Both Adams and D'Hondt apportion $7$ of the $10$ seats to party 1 and $3$ seats to party 2.  However, the inequalities used to generate the apportionment sequences show that the orders in which the seats are apportioned are different.  In particular, the inequalities
\begin{align*}
\text{Adams } c=0 \hskip 5mm & \frac{32}{0} > \frac{14}{0} > \overbrace{\frac{32}{1}> \frac{32}{2}> \frac{14}{1}> \frac{32}{3}> \frac{32}{4}> \frac{14}{2}> \frac{32}{5}> \frac{32}{6}}^{\text{same sequence of inequalities}} \\
\text{D'Hondt } c=1 \hskip 5mm & \underbrace{\frac{32}{1} > \frac{32}{2} > \frac{14}{1}> \frac{32}{3}> \frac{32}{4}> \frac{14}{2}> \frac{32}{5}> \frac{32}{6}}_{\text{same sequence of inequalities}} > \frac{14}{3}> \frac{32}{7} \end{align*}
yield the apportionment sequences
\begin{align*}
\text{Adams } c=0 \hskip 5mm & (1,\, 2,\, \overbrace{1, \, 1, \, 2, \, 1, \, 1, \, 2, \, 1, \, 1}^{\text{same sequence}}) \\
\text{D'Hondt } c=1 \hskip 5mm & (\underbrace{1, \, 1, \, 2, \, 1, \, 1, \, 2, \, 1, \, 1}_{\text{same sequence}}, \, 2, \, 1). \end{align*} Notice that the behavior after the first $2$ seats under Adams' method matches the first $8$ seats under the D'Hondt method.  This is generalized below.  The sequence also encodes the apportionment for every house size less than 10; for example, if only $8$ seats were apportioned, the sequences indicate that party 1 would receive $5$ seats under Adams and $6$ seats under D'Hondt. \end{example}

We focus on the apportionment sequences that include apportionments for all house sizes $h$ by considering the infinite sequence.

\begin{definition}  For  fixed  $\mathbf p$, cutpoint $c$,  and a house size $h$, let $S_c(h, \mathbf p) = (s_1, s_2, \dots, s_h)$ be the apportionment sequence under the stationary divisor method with cutpoint $c$ where  $s_j =i$ if party $i$  receives the $j$th seat for each $j =1, \ldots, h$.  Let $S_c( \mathbf p) = \lim_{h \rightarrow \infty} S_c( h, \mathbf p)$.   \end{definition}

Using this notation for the sequences from Example \ref{firstex}, it follows that $S_0(10, \mathbf p) =  (1, \,2, \,1, \, 1, \, 2, \, 1, \, 1, \, 2, \, 1, \, 1)$ and $S_1(10, \mathbf p) =(1, \, 1, \, 2, \, 1, \, 1, \, 2, \, 1, \, 1, \, 2, \, 1)$.  The relationship demonstrated between the apportionment sequences under Adams' and D'Hondt's methods is generalized in the following theorem.

\begin{theorem}\label{jeff_adams}
For every positive vote vector $\mathbf p = (p_1, p_2, \ldots, p_n)$, if $S_1(\mathbf p) = (s_1, s_2, \dots)$ is the sequential apportionment under D'Hondt's method, then the sequential apportionment under Adams' method is $S_0(\mathbf p) = (t_1, t_2, \dots) = (1, 2, \dots, n, s_1, s_2, s_3, \dots)$.  \end{theorem}

\begin{proof} 
We proceed by induction. First we show our base case that $t_{n+1} = s_1$.  Under Adams' method, each of the parties $1, 2,  \dots, n$ receive, in this order, one of the first $n$ seats, so that $t_i = i$ for $i = 1$ to $n$. Seat $n+1$ is given to the party maximizing $\frac{p_i}{1}$, which is party 1 because $p_1 \ge p_i$ for $i=2$ to $n$; hence, $t_{n+1} = 1$. Under D'Hondt's method, $s_1 = 1$ because $i=1$ maximizes the same ratio: $\frac{p_i}{1}$. And, $t_{n+1} = s_1$.

Assume that $t_{i+n} = s_i$ for $i = 2$ to $k$, then we show that $t_{k+n+1} = s_{k+1}$. Let $\mathbf a$ be the apportionment under the D'Hondt method for $h=k$. Then, $a_j$ is the number of $s_i = j$ for $i = 1$ to $k$.  By definition, seat $k+1$ goes to the party with the smallest index $i$ among those indices that maximize $\frac{p_i}{a_i+1}$, or equivalently,
\begin{equation*}
\min_i \argmax_{i = 1 \text{ to }n} \left\{ \frac{p_i}{a_i + 1} \right\}.
\end{equation*}
By the inductive hypothesis, $b_j = a_j + 1$ is party $j$'s apportionment under Adams' method for $h = k+n$. It follows that seat $k+n+1$ goes to party 
\begin{equation*}
\min_i \argmax_{i = 1 \text{ to }n} \left\{ \frac{p_i}{b_i} \right\} = \min_i \argmax_{i = 1 \text{ to }n} \left\{ \frac{p_i}{a_i + 1} \right\}.
\end{equation*}
Hence, $t_{k+n+1} = s_{k+1}$, completing the proof.
\end{proof}

Many  properties of apportionment methods are best described in terms of the number of seats each party receives.  Let $F_c(h, \mathbf p) = \mathbf a=(a_1, a_2, \ldots, a_n)$ indicate that after $h$ seats, the apportionment method has allocated $a_i$ seats to  party  $p_i$,  where $h=\sum_i a_i$. Thus, in Example \ref{firstex}, $F_0(10, \mathbf p) = F_1(10, \mathbf p)=(7,3)$.  The relationship between the apportionments from the Adams and the D'Hondt methods from Theorem \ref{jeff_adams} can be recast in terms of the number of seats apportioned.

\begin{corollary}
For every positive vote vector $\mathbf p = (p_1, p_2, \ldots, p_n)$, if $F_1(h,\mathbf p) = \mathbf a$, then $F_0(h+n, \mathbf p) = \mathbf a + \mathbf 1$ where $\mathbf 1$ is the vector of all 1s.
\end{corollary}

\subsection{A Refinement of the Notion of Bias}

The tendency of an apportionment method to systematically favor larger or smaller parties relative to exact proportionality is referred to as size bias. Although bias was part and parcel of the debate over the use of methods to apportion the U.S. House of Representatives in the 19th century, Huntington \cite{H1, H2} compared different methods in terms of how they optimize different measures of error from exact proportionality. Balinski and Young \cite{BY} compared divisor methods in terms of bias, providing the following definition and proposition.

\begin{definition}
A divisor method $F$ favors large parties relative to divisor method $F^\prime$ if for every $\mathbf p$ and $h$ with $F(h, \mathbf p) = \mathbf a$ and $F^\prime(h, \mathbf p) = \mathbf a^\prime$, then $p_i>p_j$ implies either $a_i \ge a_i^\prime$ or $a_j \le a_j^\prime$.
\end{definition}

\begin{proposition}\label{storder}
If $F_c$ and $F_{c^\prime}$ are stationary divisor methods corresponding to cutpoints $c$ and $c^\prime$, respectively, with $c > c^\prime$,  then $F_c$ favors large parties relative to $F_{c^\prime}$.
\end{proposition} 

The proposition implies that, among stationary methods, the D'Hondt method is the most biased in favor of large parties, while Adams' method is the most biased in favor of small parties. Indeed, it is well-known that these two methods are the most biased in terms of size among all divisor methods, not just among the class of stationary methods.  

Relative biases for stationary divisor methods can also be seen through a lexicographic ordering of their apportionment sequences.  For two sequences $s = (s_1, s_2, s_3, \ldots)$ and $t = (t_1, t_2, t_3, \ldots)$, then $s > t$ if $s_i = t_i$ for $i = 1$ to $k$ (so, the sequences agree for the first $k$ places) and $s_{k+1} > t_{k+1}$ for some $k \ge 1$.  For a fixed set of votes $p_i$ for $i= 1 $ to $n$, the following proposition shows that apportionment sequences from stationary divisor methods are lexicographically ordered by their cutpoints.

\begin{proposition}\label{ordering} If $c> c^\prime$, then $S_c(h,  \mathbf p) \le S_{c^\prime}(h, \mathbf p)$ and $S_c( \mathbf p) \le S_{c^\prime}(\mathbf p)$ for every $\mathbf p$ and $h>0$. 
 \end{proposition}

\begin{proof}
Suppose $c > c^\prime$ and that there exists $\mathbf p$ and an $h$ such that $S_c(h,  \mathbf p) > S_{c^\prime}(h, \mathbf p)$. Let  $h'+1$ be the first house size for which $S_c(h' + 1,  \mathbf p)$ and $S_{c^\prime}(h' + 1, \mathbf p)$ differ. Thus  
$$S_c(h^\prime +1, \mathbf p) = (s_1, s_2, \ldots, s_{h^\prime}, i) \quad \mbox{and} \quad S_{c^\prime}(h^\prime+1, \mathbf p)  = (s_1, s_2, \ldots, s_{h^\prime}, j)$$ 
 for some $s_1, \ldots, s_{h^\prime}$ and $i > j$.  In addition,  $F_c(h^\prime, \mathbf p) = F_{c'}(h^\prime, \mathbf p)$; let $\mathbf a$ be this common value. 

 Under $c^\prime$, party $j$ gets the $h'+1$ seat, which implies $$\frac{p_j}{a_j+c^\prime} > \frac{p_i}{a_i+c^\prime}.$$ Under $c$, party $i$ gets the $h'+1$ seat, which implies $$\frac{p_i}{a_i+c} > \frac{p_j}{a_j+c}.$$ Cross multiplying and adding these inequalities yields 
 $$cp_i+c^\prime p_j > cp_j+c^\prime p_i \iff (c-c^\prime)(p_i-p_j)>0.$$
Since $c>c^\prime$, this implies $p_i>p_j$.  But, due to the ordering of the parties based on the number of votes, then $i<j$.  This contradiction proves the claim.   
\end{proof}

By Proposition \ref{ordering}, for every $h$, $S_0(h, \mathbf p) \ge S_c (h, \mathbf p) \ge S_1(h, \mathbf p)$ for all $c \in (0,1)$. Since this holds for all $h$, it follows that $S_0(\mathbf p) \ge S_c(\mathbf p) \ge S_1( \mathbf p)$.  This lexicographical ordering of apportionment sequences offers a refined perspective on bias, showing how bias is less an absolute statement of the number of seats a party receives and more a reflection of when large or small parties receive their seats.

To achieve this refinement of bias, we need to recall additional properties (homogeneity and weak proportionality) of divisor methods. Divisor methods satisfy homogeneity because the apportionment only depends on the relative proportion of votes $p_i/P$ where $P = p_1 + \cdots + p_n$, noticeable in the definition of sequential apportionment.

\begin{definition} An apportionment method $F$ satisfies homogeneity if and only if $F(h, \mathbf p) = F(h, \lambda \mathbf p)$ for any $\lambda > 0$.
\end{definition}

\begin{definition} Let $\mathbf a = (a_1, \ldots, a_n) \in \mathbb N^n$, $p_i = \lambda a_i$ for all $i$, and $h = \sum_i a_i$.   An apportionment method $F$ satisfies weak proportionality if and only if $F(h, \mathbf p) = \mathbf a$. 
\end{definition}

In words, weak proportionality states that if $\mathbf a$ is proportional to the vote vector $\mathbf p$, then $\mathbf a$ should be the outcome of the apportionment method for $h = \sum_i a_i$.  All divisor methods satisfy weak proportionality, which is illustrated with a continuation of  Example \ref{firstex}.

\begin{example} \label{continuation}
Let $\mathbf p = (32, 14)$. Since $\mathbf a = (16, 7)$ and $\mathbf a = (32, 14)$ are both proportional to $\mathbf p$, then under D'Hondt's method,  $F_1(23, \mathbf p) = (16, 7)$ and  $F_1(46, \mathbf p) = (32, 14).$ Additionally, their
 corresponding sequences are
\begin{align*}
S_1(23, \mathbf p) &= (1^2, \, 2, \,1^2, \,2, \,1^2, \,2, \,1^3, \,2, \,1^2, \,2, \,1^2, \,2, \,1^3, \,2) \quad \mbox{and} \\
 S_1(46, \mathbf p) &= (1^2,\,2, \,1^2, \,2, \,1^2, \,2, \,1^3, \,2, \,1^2, \,2, \,1^2, \,2, \,1^3, \,2, \\
 & \hskip 6mm \,1^2, \,2, \,1^2, \,2, \,1^2, \,2, \,1^3, \,2, \,1^2, \,2, \,1^2, \,2, \,1^3, 2)
 \end{align*}
where $1^k$ denotes a sequence of $k$ 1s.
\end{example}

The proof that divisor methods satisfy weak proportionality follows most easily from an equivalent (non-sequential) definition of divisor methods. Under this definition, each apportionment method corresponds to a ``rounding rule'' $f$ that  assigns to each positive integer $k$ a value $f(k) \in [k-1, k]$. Numbers in    $[k-1, f(k))$ are rounded down  to $k -1$ and numbers in $[f(k), k]$ are rounded up to $k$. (The function must also satisfy $f(k) \ne f(k+1)$, which could only happen if both were $k$.)  The divisor method associated with rounding rule $f$ is defined as follows.

\begin{definition}  A divisor method $F$ with rounding rule $f$ satisfies $F_f(h, \mathbf p) = \mathbf a$ if there exists a nonnegative divisor $d$ such that $f(a_i) \le \frac{p_i}{d} \le f(a_i+1)$ for each $i$ where $a_i$ are nonnegative integers satisfying $\sum_{i=1}^n a_i = h$.   \end{definition}

For stationary divisor methods, $f(k) = k-1+c$ and numbers in  $[k-1, k]$ are rounded up or down depending on the whether they are above or below $k-1+c$ (hence the name cutpoint).  For these divisor methods (using the more direct notation $F_c$ rather than $F_f$), $F_c(h, \mathbf p) = \mathbf a$ if there exists a divisor $d$ such that $a_i-1+c \le \frac{p_i}{d} \le a_i+c$ for each $i$, or equivalently
$$\max_i \frac{p_i}{a_i+c} \le d \le \min_i \frac{p_i}{a_i-1+c}.$$

 The fact that stationary divisor methods can be defined in two ways---through a rounding rule and through a sequential process---is a consequence of the divisor methods' {\it house monotonicity}. An apportionment method is house monotonic if, as the number of seats increases, no party's allocation decreases. Thus, for stationary divisor methods, and divisor methods more generally, seats can either be allocated simultaneously, by finding an appropriate value of $d$, or sequentially, by increasing the house size one by one. This property is not satisfied by all apportionment methods. (Hamilton's method, or the method of greatest remainders, is the best-known example of an apportionment method that is not house monotone.)

Divisor methods satisfy weak proportionality since if $\mathbf p = k \mathbf a$ for some $k \ge 1$, then
$\max_i \frac{p_i}{f(a_i+1)} \le k \le \min_i \frac{p_i}{f(a_i)}.$ If $f(k)=k-1+c$, this is equivalent to
$$\max_i \frac{p_i}{a_i+c} \le k \le \min_i \frac{p_i}{a_i-1+c}$$
 for all $c \in [0, 1]$.

 A consequence of weak proportionality is that there are an infinite number of instances of $h$ (such as multiples of $\mathbf p$) where the apportionment agrees for all cutpoints $c$.  All stationary methods are in synch at these values of $h$. Our perspective of bias is that it gives insight as to what happens between these values of commonality.

 \subsection{Periodic Behavior}

 In Example \ref{firstex}, the smallest number of house seats $h$ for which there is an exactly proportional allocation is $23$; this follows because $p_1=32$ and $p_2=14$ are both divisible by $2$ and $(p_1+p_2)/2 = 23$.  From Example \ref{continuation}, not only is the apportionment $F_1(46, (32,14)) = 2 \cdot F_1(23, (32,14))$, but the sequence $S_1(46, (32,14))$ is the sequence $S_1(23, (32,14))$ concatenated with itself.  Below we see that this happens more generally.

Weak proportionality ensures that an apportionment method allocates seats in exact proportion to the votes received when possible. As noted before, this occurs when the house size is a multiple of the number of votes cast.  The smallest house size for which exact proportionality occurs is  $P = (p_1 + \cdots + p_n)/\gcd(p_1, p_2, \dots, p_n)$. It follows that the vote vector $\mathbf P =  \frac{(p_1, p_2, \dots, p_n)}{\gcd(p_1, p_2, \dots, p_n)}$ is an integer vector that is proportional to $\mathbf p$. Thus, by both homogeneity and weak proportionality, we have the following proposition. 

\begin{proposition} \label{prop:weakprop} If $F$ is a divisor method and $P = (p_1+ \cdots + p_n)/\gcd(p_1, \ldots, p_n)$, then $F( mP, \mathbf p) = m \cdot \mathbf P$ for all  integers $m \ge 1$.
\end{proposition}

The periodic behavior suggested in Example \ref{continuation} that $S_1(46, (32,14))$ is $S_1(23, (32,14))$ repeated twice holds in general.  For integer vote vectors $\mathbf p$, the following theorem shows that the apportionment sequence is periodic for all stationary divisor methods.

\begin{theorem} \label{periodic} For a fixed $\mathbf p$, let $P=(p_1 + \cdots + p_n)/ \gcd(p_1, \dots, p_n)$.  For $c \in [0,1]$, the stationary divisor method with cutpoint $c$ generates a periodic apportionment sequence $S_c(\mathbf p) = (s_1, s_2, \ldots)$ of period $P$ where $s_{mP+k} = s_k$ for all $k=1, \ldots, P$ and $m\ge 1$. \end{theorem}

\begin{proof}
Let $d=\gcd(p_1, \ldots, p_n)$ and recall from Proposition \ref{prop:weakprop} that if $\mathbf P = \mathbf p/d$ then $F_c(mP, \mathbf p) = m \cdot \mathbf P$ for all $m \ge 1$. Let $S_c(P, \mathbf p) = (s_1, \ldots, s_P)$ be the apportionment sequence for the first $P$ seats.  We consider which party receives seat $mP+1$. This is awarded to the minimally indexed party that maximizes the ratio $\frac{p_i}{mP+c}$.  However, the ratio $\frac{p_i}{mP+c}$ is maximized when $\frac{mP+c}{p_i} = \frac{m}{d} + \frac{c}{p_i}$ is minimized. Because $\frac{m}{d}$ is a constant, then seat $mP+1$ is awarded to the minimally indexed party that minimizes $\frac{c}{p_i}$ or maximizes $\frac{p_i}{c}$; hence $s_{mP+1}=s_1$.

Seat $mP+2$ seat is awarded to the party with the smallest index that minimizes the ratios in the set
$$\left\{ \left\{ \frac{mP+c}{p_i}\right\}_{i \ne s_i}, \frac{mP+1+c}{p_{s_1}} \right\}= \left\{ \left\{ \frac{m}{d} +\frac{c}{p_i}\right\}_{i \ne s_i}, \frac{m}{d}+\frac{1+c}{p_{s_1}} \right\}.$$
 Again, subtracting the  common term  $m/d$, this is the same party that received the $2^{nd}$ seat; hence $s_{mP+2}=s_2$. 
 
Continuing in this way, we see that $s_{mP+k}=s_k$ for $k=1, \ldots, P$.  Hence $S_c(\mathbf p)$ is periodic with period at most $P$.

Now suppose the period is equal to $Q < P$ where  $F_c(Q, \mathbf p) = \mathbf q$. Then $Q$ divides $P$. Let $\lambda \in \mathbb N$ be such an  $\lambda Q=P$. By periodicity, $F_c(P, \mathbf p) = F_c(\lambda Q, \mathbf p)= \lambda \cdot F_c(Q, \mathbf p)=\lambda \mathbf q$.  But $F_c(P, \mathbf p) =\mathbf P  =  \mathbf p/d$, which implies $ \mathbf p/d=  \lambda \mathbf q$. Thus $\lambda$ divides $p_i/d$ for each $i$, which is impossible unless $\lambda =1.$ Hence, $Q=P$.
\end{proof}

Although we will usually assume that the $p_i$ are natural numbers, so that $p_i/p_j$ is rational, divisor methods can be applied for irrational values, too. One can imagine contexts in which irrational values could arise in the JIT sequencing context.  However, when $p_i/p_j$ is irrational for some $i$ and $j$, the apportionment sequence is not periodic.

\begin{proposition}
If $p_i/p_j$ is irrational for some $i$ and $j$, then the apportionment sequence is not periodic under any stationary divisor method. 
\end{proposition}
\begin{proof}
Let $c$ be the cutpoint for a stationary divisor method.  Assume that the apportionment sequence $S_c(\mathbf p)$ is periodic with period $P$ for vote vector $\mathbf p$ with $p_i/p_j$ irrational.  Because of the periodicity, if $F_c(P, \mathbf p) = (a_1, \dots, a_n)$, then $F_c(mP, \mathbf p) = (ma_1, \ldots, ma_n)$ for all positive integers $m$. 

Fix $m$  and consider how the $mP+1$ seat would be assigned. If $p_i > p_j$, then party $i$ receives $ma_i + 1$ seats before party $j$ receives $m a_j + 1$ seats, which implies that 
$$\frac{p_i}{ma_i + c} > \frac{p_j}{ma_j  + c} \text{ or }\frac{p_i}{p_j} > \frac{ma_i + c}{ma_j + c}.$$
Taking the limit as $m \rightarrow \infty$, $p_i/p_j \ge a_i/a_j$.

On the other hand, party $j$ receives $ma_j$ seats before party $i$ receives $ma_i + 1$ seats, leading to the inequalities
$$\frac{p_j}{ma_j -1 + c} > \frac{p_i}{ma_i + c} \text{ and } \frac{ma_i + c}{ma_j -1 + c} > \frac{p_i}{p_j}.$$ Again, taking the limit as $m \rightarrow \infty$ implies $a_i/a_j \ge p_i/p_j$. 

It follows that the ratio $p_i/p_j = a_i/a_j$ is rational, which is a contradiction.  Hence, the apportionment sequence cannot be periodic. \end{proof}

Many divisor methods are based on means, in which the cutpoint between $k-1$ and $k$ is a mean of $k-1$ and $k$.  So, the cutpoint $f(k)$ could be the geometric mean (as in the Hill-Huntington method), the harmonic mean (as in Dean's method), or any of the parameterized power means.  In these cases, the resulting sequence need not be periodic, as demonstrated by the following example. Key here is that all divisor methods satisfy weak proportionality.

 \begin{example}  Let $\mathbf p = (23, 4)$. Hill-Huntington's method corresponds to the rounding rule $f(k) = \sqrt{(k-1)k}$, where, sequentially, the compared ratios are of the form $p_i / \sqrt{a_i\left(a_i+1\right)}$.    Because Hill-Huntington's method is weakly proportional, the apportionment for $27m$ seats is $(23m, 4m)$ for all positive integers $m$. However the first block of $27$ seats is awarded in a different sequence from subsequent blocks of $27$.  The apportionment sequence under the Hill-Huntington method for $h=27$ is
\begin{equation*}
S_{H}(27, \mathbf p) = (1, \,2, \,1^7, \,2, \,1^6, \,2, \,1^6, \, 2, \,1^3).
\end{equation*}
 Subsequent blocks of $27$ seats match the sequential allocation of the Sainte-Lagu\"{e} method, which is
 $$S_{1/2}(27, \mathbf p) = (1^3, \, 2, \, 1^6, \, 2, \, 1^5, \, 2, \, 1^6, \, 1^3) \ne S_{H}(27, \mathbf p).$$
 Consequently, $S_H(\mathbf p)$ is $S_{H}(27, \mathbf p)$ appended with an infinite number of copies of $S_{1/2}(27, \mathbf p)$.
This follows because as $h$ increases, the Hill-Huntington method approaches the Sainte-Lagu\"{e} method. To see this, we rewrite $\sqrt{x(x+1)}$ as $x \sqrt{1+\frac{1}{x}}$ and use the Taylor series to show that the difference between the arithmetic mean $\frac{x + (x+1)}{2}$ and the geometric mean $\sqrt{x(x+1)}= x \sqrt{1+\frac{1}{x}}$ goes to zero as $x$ increases;  specifically, 
\begin{align*}
\lim_{x \rightarrow \infty}  x + \frac{1}{2} - \sqrt{x(x+1)} &= \lim_{x \rightarrow \infty} x + \frac{1}{2} -  x\left(1+\frac{1}{2 x}-\frac{1}{8 x^2}+O\left(x^{-3}\right)\right) \\
&=\lim_{x \rightarrow \infty} \frac{1}{8 x}-O\left(x^{-2}\right) = 0. 
\end{align*}
\end{example}

In this paper, our primary interest  is in understanding the periodic sequences that arise from stationary divisor methods when the $p_i$ are rational. In  Section \ref{2parties},  we determine explicit expressions for all possible sequences when $n=2$; we extend this analysis to more than two parties in Section \ref{more_parties}.

\section{Apportionment Sequences for Two Parties}\label{2parties}

For two parties with vote totals  $p_1>p_2$, Theorem \ref{periodic} states that the apportionment sequence from any stationary divisor is periodic with length $(p_1+p_2)/\gcd(p_1, p_2)$.  The exact sequence depends on the value of $c \in [0, 1]$.  

To understand what these sequences look like,  we assume first that  $p_1$ and $p_2$ are relatively prime, so that $\gcd(p_1,p_2) = 1$.  By weak proportionality and homogeneity, the apportionment sequence consists of blocks of period length $p_1 + p_2$ and each block consists of $p_1$ 1s and $p_2$ 2s.  The following lemma places constraints on what sequences are possible.

\begin{lemma} \label{two-in-a-row} For a 2-party apportionment sequence under a stationary divisor method, if $p_1>p_2$, then party 2 never receives two seats in a row. \end{lemma}

\begin{proof} Fix $0 \leq c \leq 1$. For the stationary divisor method with cutpoint $c$ and a house size of $h$, assume the apportionment is $\mathbf a = (a_1,a_2)$, so that $h = a_1 + a_2$.  Further, assume that party 1 receives the next seat followed by two seats for party 2.  By this assumption, the following inequalities must hold:
\begin{equation*} \frac{p_1}{a_1+c} > \frac{p_2}{a_2+c}, \quad   \frac{p_2}{a_2+c} > \frac{p_1}{a_1 +1 +c} \quad \mbox{and} \quad \frac{p_2}{a_2+1+c} > \frac{p_1}{a_1+1+c}.\end{equation*}

Respectively, the first and third inequalities imply the inequalities
\begin{equation*} 
\frac{a_1p_2 - a_2p_1}{p_1-p_2} < c < \frac{a_1p_2-a_2p_1 - (p_1 - p_2)}{p_1-p_2}.
\end{equation*}
Because $p_1 - p_2 > 0$, the right side of the above inequality is less than the left side, providing a contradiction. Thus, no such $c$ can exist and party 2 cannot get two seats in a row from a stationary divisor method.\end{proof}

Combining Lemma \ref{two-in-a-row} and Theorem \ref{periodic}, this means that any two-party apportionment sequence from a stationary divisor method will have a repeated pattern of the form 
$$S_c(p_1+p_2,(p_1, p_2))=\underbrace{1, 1, \ldots, 1}_{k_1 \text{ times}}, 2, \underbrace{1, 1, \ldots, 1}_{k_2 \text{ times}}, 2, \ldots, \underbrace{1, 1, \ldots, 1}_{k_{p_2} \text{ times}}, 2, \underbrace{1, 1, \ldots, 1}_{p_1- \sum_i k_i \text{ times}},$$
where the last block of 1s of length $p_1- \sum_i k_i$ may be zero.  As in Section \ref{prelims}, we use the notation 
 $$S_c(p_1+p_2, (p_1, p_2)) = 1^{k_1}, \, 2, \, 1^{k_2}, \, 2,  \, \ldots \, , 1^{k_{p_2}}, \, 2, \, 1^{p_1-(k_1+ \cdots + k_{p_2})},$$
\noindent where the values of the $k_i$ depend on $p_1$ and $p_2$ as well as on the cutpoint $c$. This is  illustrated in the following example.

 \begin{example}
 Suppose $p_1=16$ and $p_2=7$. For each $c \in [0, 1]$, the apportionment sequence has period $23$. If $c=0$, the first 23 elements of the sequence are
  $$1,  \hskip 1mm 2, \hskip 1mm 1^2, \hskip 1mm 2, \hskip 1mm 1^2, \hskip 1mm 2, \hskip 1mm 1^2, \hskip 1mm 2, \hskip 1mm 1^3, \hskip 1mm 2, \hskip 1mm 1^2, \hskip 1mm 2, \hskip 1mm 1^2, \hskip 1mm 2, \hskip 1mm 1^2.$$
For $c$ close to 0, the apportionment sequence remains the same. As $c$ increases in the interval $[0,1]$, there are 10 possible apportionment sequences, arranged in increasing lexicographical order. These are indicated in Table  \ref{first-sequential-example}. The changes between lines are indicated with the colors red and blue.

\begin{table}[tbh]
\caption{Apportionment sequences for $p_1 = 16$ and $p_2 = 7$ for stationary divisor methods dependent on $c$.}
\label{first-sequential-example}
\renewcommand{\arraystretch}{1.2}
\setlength{\tabcolsep}{2pt}
\begin{tabular}{c|*{15}c}
$c$ & \multicolumn{14}{c}{Sequence for $p_1 = 16$ and $p_2 = 7$}  \\
\hline
$\left[ 0, \frac{1}{9} \right)$ & $1$ & $2$ & $1^2$ & $2$ & $1^2$ & $2$ & $1^2$ & $2$ & \textcolor{red}{$1^3$} & $2$ & $1^2$ & $2$ & $1^2$ & $2$ & $1^2$ \\
$\left[ \frac{1}{9}, \frac{2}{9} \right)$ & $1$ & $2$ & $1^2$ & $2$ & $1^2$ & $2$ & \textcolor{red}{$1^3$} & $2$ & $1^2$ & $2$ & $1^2$ & $2$ & $1^2$ & $2$ & \textcolor{blue}{$1^2$} \\
$\left[ \frac{2}{9}, \frac{3}{9} \right)$ & $1$ & $2$ & $1^2$ & $2$ & $1^2$ & $2$ & \textcolor{red}{$1^3$} & $2$ & $1^2$ & $2$ & $1^2$ & $2$ & \textcolor{blue}{$1^3$} & $2$ & $1$ \\
$\left[ \frac{3}{9}, \frac{4}{9} \right)$ & $1$ & $2$ & $1^2$ & $2$ & \textcolor{red}{$1^3$} & $2$ & $1^2$ & $2$ & $1^2$ & $2$ & $1^2$ & $2$ & \textcolor{blue}{$1^3$} & $2$ & $1$ \\ 
$\left[ \frac{4}{9}, \frac{5}{9} \right)$ & $1$ & $2$ & $1^2$ & $2$ & \textcolor{red}{$1^3$} & $2$ & $1^2$ & $2$ & $1^2$ & $2$ & \textcolor{blue}{$1^3$} & $2$ & $1^2$ & $2$ & $1$ \\ 
$\left[ \frac{5}{9}, \frac{6}{9} \right)$ & $1$ & $2$ & \textcolor{red}{$1^3$} & $2$ & $1^2$ & $2$ & $1^2$ & $2$ & $1^2$ & $2$ & \textcolor{blue}{$1^3$} & $2$ & $1^2$ & $2$ & $1$ \\
$\left[ \frac{6}{9}, \frac{7}{9} \right)$ & $1$ & $2$ & \textcolor{red}{$1^3$} & $2$ & $1^2$ & $2$ & $1^2$ & $2$ & \textcolor{blue}{$1^3$} & $2$ & $1^2$ & $2$ & $1^2$ & $2$ & $1$ \\
$\left[ \frac{7}{9}, \frac{8}{9} \right)$  & \textcolor{red}{$1^2$} & $2$ & $1^2$ & $2$ & $1^2$ & $2$ & $1^2$ & $2$ & \textcolor{blue}{$1^3$} & $2$ & $1^2$ & $2$ & $1^2$ & $2$ & $1$ \\
$\left[ \frac{8}{9}, \frac{9}{9} \right)$  & $1^2$ & $2$ & $1^2$ & $2$ & $1^2$ & $2$ & \textcolor{blue}{$1^3$} & $2$ & $1^2$ & $2$ & $1^2$ & $2$ & $1^2$ & $2$ & \textcolor{red}{$1$} \\
$1$  & $1^2$ & $2$ & $1^2$ & $2$ & $1^2$ & $2$ & $1^3$ & $2$ & $1^2$ & $2$ & $1^2$ & $2$ & \textcolor{red}{$1^3$} & $2$ &  
\end{tabular} 
\end{table} 
\end{example}

 Note that the apportionment sequences are constant for $c$ in intervals of the form $\left[\frac{\ell}{p_1-p_2}, \frac{\ell+1}{p_1-p_2}\right)$ (the half-open intervals are due to  the tie-breaking rule); there is also an apportionment sequence in which there is no final 1 corresponding to $c=1$.  Additionally,  $k_1, k_2, \ldots, k_7 \in \{2, 3\}$.  For the sequence to remain roughly proportional, party 1 must receive 2 or 3 seats for every seat received by party 2 because $2 < \frac{23}{7} < 3$.
  
 As a first step towards  generalizing these observations, we identify necessary and sufficient conditions on the values of $k_i$ for an apportionment sequence to be the result of a stationary divisor method.

\begin{proposition}\label{prop:increasing}
Let $p_1$ and $p_2$ be relatively prime with $p_1>p_2$.  Assume an apportionment sequence of length $p_1 + p_2$ has the form 
$$1^{k_1}, \, 2, \,  1^{k_2}, \, 2,  \, \ldots \, , 1^{k_j}, \, 2, \, \ldots , \, 1^{k_{p_2}},  \, 2, \, 1^{p_1-(k_1+\cdots +k_{p_2})}.$$ 
\begin{enumerate} 
\item The apportionment sequence is generated by a stationary divisor method if and only if
   \begin{equation}\label{eq_k1} 1 \le k_1 \le \left\lfloor \frac{p_1}{p_2} \right\rfloor \end{equation} and
   \begin{equation}\label{eq_ki} \left\lfloor\frac{(i^\prime-i)p_1}{p_2}\right\rfloor \le k_{i+1}+ \cdots + k_{i^\prime} \le  \left\lfloor \frac{(i^\prime-i)p_1}{p_2}\right\rfloor + 1 \mbox{ for all } 1 \le i < i^\prime \le p_2. \end{equation}

\item If these conditions from inequalities (\ref{eq_k1}) and (\ref{eq_ki}) are met, then the cutpoint $c$ lies in the interval 
\begin{equation} \label{c_interval}
 \left[ \max_{i \ge 1} \frac{p_2(\sum_{j \le i} k_j-1) -p_1(i-1)}{p_1-p_2}, \,\, \min_{i \ge 1} \frac{p_2(\sum_{j \le i} k_j)-p_1(i-1)}{p_1-p_2}\right).\end{equation}
Moreover, if  $\delta_c =  \frac{(p_1-p_2)}{p_2}c$, then 
\begin{equation}
\begin{aligned}
 k_1 &=\lfloor  \delta_c \rfloor + 1, \\
 k_i & =\left\lfloor \frac{p_1}{p_2}(i-1) + \delta_c \right\rfloor - \left\lfloor \frac{p_1}{p_2}(i-2) + \delta_c \right\rfloor \quad \mbox{ for } i=2, \ldots, p_2, \text{ and}\\
k_{p_2+1} &= p_1 - \left\lfloor  \frac{p_1}{p_2}(p_2-1)+\delta_c \right\rfloor -1.\\
\end{aligned} 
\label{ki_formulas}
\end{equation}
\end{enumerate}
\end{proposition}

 \begin{proof}
We prove  part {\it 2}. The proof of part {\it 1} is in the Appendix.  Suppose the apportionment sequence satisfies the conditions in inequalities (\ref{eq_k1}) and (\ref{eq_ki}) and that it is generated by a stationary divisor method with cutpoint $c$. Then for each $i= 1, \ldots, p_2$, since party 1 receives $k_1 + \cdots +k_i=\sum_{j \le i} k_j$ seats before party 2 receives $i$ seats, $c$ must satisfy (using the tie-breaking rule),
\begin{equation}\label{kratios}
\frac{p_1}{\sum_{j \le i}k_j -1+c} \ge \frac{p_2}{i-1+c}  > \frac{p_1}{\sum_{j \le i}k_j +c}
\end{equation}
or
\begin{equation}\label{cratios}
\frac{p_2(\sum_{j \le i}k_j -1)-p_1(i-1)}{p_1-p_2} \le c < \frac{p_2(\sum_{j \le i}k_j) -p_1(i-1)}{p_1-p_2}.
\end{equation}Thus, $c$ must lie in the interval from Equation (\ref{c_interval}).

To show that  the $k_i$ satisfy the equations in (\ref{ki_formulas}), we solve for $\sum_{j \le i} k_j$ in (\ref{kratios}) to obtain
$$\frac{(p_1-p_2)c}{p_2} + \frac{p_1}{p_2}(i-1)  <   \sum_{j \le i}k_j \le \frac{(p_1-p_2)c}{p_2} + \frac{p_1}{p_2}(i-1) +1.$$
Because the left and right sides of the above inequality differ by $1$, then $\sum_{j \le i}k_j = \left\lfloor \frac{(p_1-p_2)c}{p_2} + \frac{p_1}{p_2}(i-1) \right\rfloor +1$ for all $i=1, \ldots, p_2$.  Substituting $i=1$, $i=2$, etc. sequentially into this expression we obtain the desired result. \end{proof}

We use Proposition \ref{prop:increasing} to obtain  precise expressions for the exponents in the following theorem, whose proof is in the Appendix.

  \begin{theorem}\label{amax}
 Suppose  $p_1$ and $p_2$ are relatively prime and that $p_1=ap_2+b$ for some integers $a \ge 1$ and $0< b <p_2$. If $\delta_c =  \frac{(p_1-p_2)}{p_2} c= (a-1+\frac{b}{p_2})c$, then
 \begin{align*}\label{ki_a}
k_1 & = \lfloor \delta_c \rfloor + 1,\\
k_i  & =\left\{ \begin{array}{cl} a+1 & \mbox{ if } i \in S   \\  
a & \mbox{ if }  i \notin S  \end{array} \right. , \mbox{ and}\\
k_{p_2+1} &=a+b-1 - \left\lfloor \frac{b}{p_2} (p_2-1)+ \delta_c \right\rfloor
\end{align*}
where  $$S = \left\{i = 2, \ldots, p_2 : i = \left\lceil (j- \delta_c+ \left\lfloor  \delta_c \right\rfloor)\frac{p_2}{b} \right\rceil + 1 \mbox{, for  } j = 1, 2, \ldots, b\right\}.$$
 \end{theorem}
 
 We illustrate Theorem \ref{amax} for $(p_1, p_2)=(16,7)$ and two different values of $c$.
 
 \begin{example}
Let $(p_1, p_2)=(16,7)$. Since $16 = 2 \cdot 7+2$, then $\delta_c =\left(2-1+ \frac{2}{7}\right)c = \frac{9}{7}c$ and 
$$S = \left\{i = 2, \ldots, 7 : i = \left\lceil \left(j- \textstyle\frac{9}{7}c+ \left\lfloor  \textstyle\frac{9}{7}c \right\rfloor \right)\textstyle\frac{7}{2} \right\rceil + 1 \mbox{, for  } j = 1, 2\right\}.$$
If $c=0$, then $\delta_0 = \lfloor \delta_0 \rfloor=0$. Moreover, since
$$ \left\lceil (1- 0+ \lfloor  0 \rfloor)\textstyle\frac{7}{2} \right\rceil+1 = 5 \quad \mbox{and} \quad   \left\lceil (2-0+ \lfloor  0 \rfloor)\textstyle\frac{7}{2} \right\rceil+1=8>7, $$
then $S = \{5\}$. 
  Applying the formulas in Theorem \ref{amax}, we obtain
 $$k_1  = \lfloor  0 \rfloor + 1=1, \quad  k_5=3, \quad k_2, k_3, k_4, k_6, k_7=2,  \quad \mbox{ and } \quad  k_{8} =3 - \left\lfloor \textstyle\frac{12}{7}+0\right\rfloor= 2.$$
This is the sequence that appears in the first line of Table \ref{first-sequential-example}.

If $c=\frac{6}{9}$ then $\delta_{\frac{6}{9}}=\frac{6}{7}$ and  $\lfloor \delta_{\frac{6}{9}} \rfloor =0$. Additionally, 
$$ \left\lceil \left(1- \textstyle\frac{6}{7}+ 0 \rfloor \right)\textstyle\frac{7}{2} \right\rceil+1 = 2 \quad \mbox{and} \quad   \left\lceil \left(2-\textstyle\frac{6}{7}+ \lfloor  0 \rfloor \right)\textstyle\frac{7}{2} \right\rceil+1=5, $$
so $S = \{2, 5\}$ and
 $$k_1  = \lfloor  0 \rfloor + 1=1, \quad  k_2, k_5=3, \quad k_3, k_4, k_6, k_7=2,  \quad \mbox{ and } \quad  k_{8} =3 - \left\lfloor \textstyle\frac{12}{7}+\textstyle\frac{6}{7}\right\rfloor= 1.$$
This sequence appears in line 7 of Table \ref{first-sequential-example}.
 \end{example}

One  consequence of Theorem \ref{amax} is that if $p_1=p_2+1$, then the apportionment sequences for all stationary divisor methods with $c \in [0, 1)$ coincide. 

\begin{proposition}\label{diff_one}
If $p_1=p_2+1$ and $c \in [0, 1)$, then the apportionment sequence under the stationary divisor method with cutpoint $c$ has repeated pattern
\begin{equation*} \underbrace{1, \,2,\,1, \, 2, \, \ldots,}_{p_2 \text{ times}} 1.
\end{equation*} \end{proposition}

\begin{proof} Suppose $p_1=p_2+1$ for some $p_2 \ge 1$ and $c \in [0, 1)$. From Theorem \ref{amax}, $\delta_c = (1-1+\frac{1}{p_2})c = \frac{c}{p_2}$ and $\lfloor \delta_c \rfloor  = 0$, implying $k_1=1$. Moreover,  substituting $j=1$ into the definition of the set $S$, we have  $ \lceil (1- \delta_c+\lfloor \delta_c \rfloor)\frac{p_2}{1}\rceil+1 =\lceil (1- \frac{c}{p_2})p_2 \rceil + 1 = p_2 + 1$, and so $S= \emptyset.$  Thus, $k_2= \cdots = k_{p_2}=1$ and $k_{p_2+1} = 1-\lfloor \frac{1}{p_2}(p_2-1)+\frac{c}{p_2} \rfloor = 1$, leading to the repeated pattern $1, \, 2, \, 1, \, 2, \ldots, \, 1$.
\end{proof}

When $p_1>p_2+1$, Theorem \ref{amax} can be used to determine subintervals for the values of $c \in [0, 1]$ for which the apportionment sequence is constant. As suggested by the patterns in Table \ref{first-sequential-example}, these occur on intervals of length $1/(p_1-p_2)$.  This is described in the next theorem, the proof of which appears in the appendix.

\begin{theorem}\label{constant_intervals}
Suppose $p_1$ and $p_2$  are relatively prime and $p_1>p_2+1$. Then the apportionment sequences  corresponding to stationary divisor methods are constant for  cutpoints in  the intervals  $c \in  [\frac{\ell}{p_1-p_2}, \frac{\ell+1}{p_1-p_2})$ where $\ell= 0, \ldots, p_{1}-p_2-1$.  Moreover, each of these $p_1-p_2$ sequences is distinct.
\end{theorem}

Since the apportionment is constant on each interval of the form $c \in [\frac{\ell}{p_1-p_2}, \frac{\ell+1}{p_1-p_2})$, we can substitute $c=\frac{\ell}{p_1-p_2}$ into the expressions in Theorem \ref{amax}  to get the following simplified expressions for the $k_i$ values.

\begin{corollary}\label{corsimple}
Suppose  $p_1$ and $p_2$ are relatively prime with   $p_1>p_2+1$ and $p_1=ap_2+b$ for some integers $a \ge 1$ and $0< b <p_2$. If $c \in [\frac{\ell}{p_1-p_2}, \frac{\ell+1}{p_1-p_2})$ with $\ell = \nu p_2+\ell^\prime$, then, for $i=2, \ldots, p_2$,  \begin{align*}
k_1 & = \nu+1, \\
k_i  & =\left\{ \begin{array}{cl} a+1 & \mbox{ if } i = \lceil \frac{jp_2-\ell^\prime}{b} \rceil +1 \mbox{ for  some } j = 1, 2, \ldots, b, \\
a & else,  \end{array} \right. \text{ and }\\
k_{p_2+1}& = p_1-(k_1+ \cdots + k_{p_2}).
\end{align*}
\end{corollary}

Corollary \ref{corsimple} implies that the values of $k_2, \ldots, k_{p_2}$ depend only on $\ell^\prime$ (and $b$), and  do not change for different values of $k_1=\nu+1$.  Since $\ell^\prime$ ranges from $0$ to $p_2-1$, there will be  at most $p_2$  different sequences
of the exponents $k_2, \ldots, k_{p_2}$. If  $p_1> 2p_2$, so that $p_1-p_2 > p_2$, then these sequences will repeat (as $c$ increases), with different values of $k_1$, as illustrated in the following example.

\begin{example} Suppose $p_2=7$ and $p_1=7a+2$ for some integer $a$. Then for each value of $k_1$, there are 7 different apportionment sequences corresponding to $c \in \left[\frac{\ell}{p_1-7}, \frac{\ell+1}{p_1-7}\right)$ where $\ell = \nu 7+\ell^\prime= (k_1-1)7+\ell^\prime$, as shown in Table  \ref{tab:m7}.  For each $\ell^\prime = 0, \ldots , 6$, the position of the exponents for which $k_i=a+1$ can be determined by evaluating $i_j = \left\lceil \frac{7j-\ell^\prime}{2} \right\rceil+ 1$.  Thus, when $\ell^\prime=0$ (the first row in the table),  this occurs when $i = \left\lceil \frac{7(1)-0}{2} \right\rceil+ 1= 5$. When $\ell^\prime=1$, this occurs when $i = \left\lceil \frac{7(1)-1}{2}\right\rceil + 1=4$, and so on.

\begin{table}[tbh]
 \caption{Apportionment sequences for $p_1=7a+2$ and $p_2 = 7$.}
 \label{tab:m7}
\renewcommand{\arraystretch}{1.65}
\setlength{\tabcolsep}{2pt}
\begin{tabular}{c|*{15}c}
$c$ & \multicolumn{14}{c}{Sequences for $p_1=7a+2$ and $p_2 = 7$}  \\
\hline 
             $\left[ \frac{7k_1-7}{p_1-7}, \frac{7k_1-6}{p_1-7} \right)$ & $1^{k_1}$ & $2$ & $1^a$ & $2$ & $1^a$ & $2$ & $1^a$ & $2$ & \textcolor{red}{$1^{a+1}$} & $2$ & $1^a$ & $2$ & $1^a$ & $2$ & $1^{a-k_1-1}$ \\
$\left[ \frac{7k_1-6}{p_1-7}, \frac{7k_1-5}{p_1-7} \right)$ & $1^{k_1}$ & $2$ & $1^a$ & $2$ & $1^a$ & $2$ & \textcolor{red}{$1^{a+1}$} & $2$ & $1^a$ & $2$ & $1^a$ & $2$ & $1^a$ & $2$ & \textcolor{blue}{$1^{a-k_1-1}$} \\
$\left[ \frac{7k_1-5}{p_1-7}, \frac{7k_1-4}{p_1-7} \right)$ & $1^{k_1}$ & $2$ & $1^a$ & $2$ & $1^a$ & $2$ & \textcolor{red}{$1^{a+1}$} & $2$ & $1^a$ & $2$ & $1^a$ & $2$ & \textcolor{blue}{$1^{a+1}$} & $2$ & $1^{a-k_1}$ \\
$\left[ \frac{7k_1-4}{p_1-7}, \frac{7k_1-3}{p_1-7} \right)$ & $1^{k_1}$ & $2$ & $1^a$ & $2$ & \textcolor{red}{$1^{a+1}$} & $2$ & $1^a$ & $2$ & $1^a$ & $2$ & $1^a$ & $2$ & \textcolor{blue}{$1^{a+1}$} & $2$ & $1^{a-k_1}$ \\ 
$\left[ \frac{7k_1-3}{p_1-7}, \frac{7k_1-2}{p_1-7} \right)$ & $1^{k_1}$ & $2$ & $1^a$ & $2$ & \textcolor{red}{$1^{a+1}$} & $2$ & $1^a$ & $2$ & $1^a$ & $2$ & \textcolor{blue}{$1^{a+1}$} & $2$ & $1^a$ & $2$ & $1^{a-k_1}$ \\ 
$\left[ \frac{7k_1-2}{p_1-7}, \frac{7k_1-1}{p_1-7} \right)$ & $1^{k_1}$ & $2$ & \textcolor{red}{$1^{a+1}$} & $2$ & $1^a$ & $2$ & $1^a$ & $2$ & $1^a$ & $2$ & \textcolor{blue}{$1^{a+1}$} & $2$ & $1^a$ & $2$ & $1^{a-k_1}$ \\
$\left[ \frac{7k_1-1}{p_1-7}, \frac{7k_1}{p_1-7} \right)$ & $1^{k_1}$ & $2$ & \textcolor{red}{$1^{a+1}$} & $2$ & $1^a$ & $2$ & $1^a$ & $2$ & \textcolor{blue}{$1^{a+1}$} & $2$ & $1^a$ & $2$ & $1^a$ & $2$ & $1^{a-k_1}$ \\
 \end{tabular} 
\end{table}
\end{example}

When $p_1=16$, we retrieve the results in Table \ref{first-sequential-example},  where all 7 sequences from Table \ref{tab:m7} occur for $k_1=1$ and the first 2 sequences from Table \ref{tab:m7}  occur for $k_1=2$.  

Table  \ref{first-sequential-example} also includes a row for $c=1$ (D'Hondt's method).  It is the only sequence that does not end in a ``1''.  In fact, this is true more generally.

 \begin{proposition} \label{prop:c1}
Suppose $p_1$ and $p_2$  are relatively prime and the apportionment sequence for $h = p_1 + p_2$ corresponding to the stationary divisor method with cutpoint $c$ is
 $$1^{k_1}, \,  2, \,1^{k_2}, \,  2, \,  \ldots, 1^{k_j}, \,  2, \, 1^{k_{p_2}}, \, 2 , \, 1^{p_1-(k_1+\cdots + k_{p_2})}.$$ 
Then $p_1=k_1+ \cdots + k_{p_2}$ if and only if $c=1$.
\end{proposition}
\begin{proof}
Substituting $i=p_2$ into the inequality (\ref{kratios}) in the proof of Proposition \ref{prop:increasing}, we get 
\begin{equation}\label{lastk}
\frac{p_1}{\sum_{j \le p_2} k_j-1+c}  \ge \frac{p_2}{p_2-1+c}>   \frac{p_1}{\sum_{j \le p_2} k_j+c}.\end{equation}
If $c=1$, then $ \frac{p_2}{p_2-1+c}=1,$ which implies
$$\frac{p_1}{\sum_{j \le p_2} k_j} \ge 1 >   \frac{p_1}{\sum_{j \le p_2} k_j+1}.$$
Hence, $\sum_{j \le p_2} k_j\le p_1 < \sum_{j \le p_2} k_j+1$, which implies $p_1 =\sum_{j \le p_2} k_j$.

Alternatively, if  $p_1 =\sum_{j \le p_2} k_j$, then the left-hand inequality  in (\ref{lastk}) implies
$$\frac{p_1}{p_1-1+c}  \ge \frac{p_2}{p_2-1+c} \, \Rightarrow \, c \ge \frac{p_1-p_2}{p_1-p_2}=1.$$ \end{proof}

Note that if $c=1$ and $p_1=k_1+ \cdots + k_{p_2}$, then the left-hand inequality  in (\ref{lastk}) is an equality: $\frac{p_1}{p_1-1+1} = \frac{p_2}{p_2-1+1}=1.$ Thus, the last two seats for each party are awarded, first to $p_1$ and then to $p_2$ using the tie-breaking rule.

Applying Proposition \ref{prop:c1} and substituting $c=1$ into the equations in Theorem \ref{amax} allows us to arrive at an explicit expression for the apportionment sequence corresponding to D'Hondt's method.

\begin{corollary}
If $p_1$ and $p_2$  are relatively prime with $p_1=ap_2+b$, then  the apportionment sequence for $h = p_1 + p_2$  corresponding to D'Hondt's method is
$$ 1^a, \, 2, \, 1^{k_2}, \, 2, \, 1^{k_3}, \, 2, \, \ldots,  1^{k_{p_2}}, \, 2$$
where $k_i = a+1$ if $i \in S=\left\{\lceil \frac{p_2}{b} \rceil, \lceil \frac{2p_2}{b} \rceil, \ldots, \lceil \frac{(b-1)p_2}{b} \rceil, \lceil  p_2\rceil\right\}$ and $k_i = a$ otherwise. \end{corollary}

Since $c=1$ corresponds to a different apportionment sequence from any $c \in [0,1)$, Proposition \ref{constant_intervals} implies the following.

 \begin{corollary} \label{pair-count}
If $p_1$ and $p_2$  are relatively prime, then there are exactly $p_1-p_2+1$ different apportionment sequences corresponding to stationary divisor methods.
\end{corollary}

Finally, we note  we can also use the expressions in Corollary \ref{corsimple} to work backwards from an apportionment sequence to find its corresponding cutpoint. The following theorem, whose proof follows directly from solving for $\nu$ and $\ell$ in Corollary \ref{corsimple}, provides an explicit formula for the interval that $c$ must lie in given an apportionment sequence with the appropriate form.

\begin{theorem}\label{prop:backwards}
Suppose $p_1$ and $p_2$  are relatively prime with $p_1>p_2+1$  and the apportionment sequence has the repeated pattern of 
$$1^{k_1}, \,   2, \, 1^{k_2}, \,  2, \,  \ldots, 1^{k_{p_2}}, \, 2, \, 1^{p_1-(k_1+\cdots +k_{p_2})}.$$ 
where the $k_i$ satisfy (\ref{eq_k1}) and (\ref{eq_ki}).
Then the apportionment sequence is generated by a stationary divisor method with cutpoint $c \in [\frac{\ell}{p_1-p_2}, \frac{\ell+1}{p_1+p_2})$ where  $\ell=\nu p_2+\ell^\prime$,  $\nu = k_1-1$ and $\ell^\prime = \max_{i_j \in S} \{ p_2j-(i_j-1)b \}$ where $S=\{i_j   : 2 \le i_j \le p_2 \text{ and }k_{i_j} = \lfloor \frac{p_1}{p_2} \rfloor+1 \}$. \end{theorem}

We conclude this section by remarking that all the results of this section apply to more general values of $\mathbf p=(p_1, p_2)$ with $p_1$ and $p_2$ replaced by $p_1/\gcd(p_1, p_2)$ and   $p_2/\gcd(p_1, p_2)$ in all expressions for $k_i$, etc.

\section{Apportionment Sequences for More Than Two Parties}\label{more_parties}

The properties of apportionment sequences for two parties carry over to apportionment sequences for more than two parties because stationary divisor methods are consistent.
Consistency requires that the apportionment between two parties is unaffected by the removal of another party. More precisely, suppose $h$ seats are allocated among $n$ parties according to $F(h, \mathbf p) = \mathbf a$ and suppose $N^\prime$ is a subset of parties whose total number of seats in this allocation is $\sum_{i \in N^\prime} a_i = h^\prime$. If we remove these seats and reapportion the remaining $h-h^\prime$ seats among the parties not in $N^\prime$,  their allocation is unchanged.

All divisor methods satisfy consistency because they are based on ratios of the form $p_i/p_j$ which are unchanged if additional parties are added or eliminated \cite{BY}. From a sequential apportionment point of view,   
consistency implies that given an apportionment sequence among $n$ parties, if we extract only the entries corresponding to parties $i$ and $j$, the sequence of $i$'s and $j$'s will be the same as the sequence of $1$s and $2$s generated by apportioning seats only between those two parties.

Because of this, some of the results for 2-party sequences extend easily to more than two parties.  Lemma \ref{two-in-a-row} is generalized in the following proposition.

\begin{proposition} \label{two-in-a-row_n} For an $n$-party apportionment sequence under a stationary divisor method, if $p_k > p_\ell$ and party $k$ receives a seat in the sequence, then party $\ell$ cannot receive two seats before party $k$ receives another seat. \end{proposition}

\begin{proof}
For any divisor method, the subsequence with $k$ and $\ell$ as its terms is the same as the 2-party sequence when $k$ is identified with 1 and $\ell$ is identified with 2.  By Lemma \ref{two-in-a-row}, the sequence with 1s and 2s cannot have two 2s in a row.  Thus,  two $\ell$'s cannot appear between consecutive $k$'s.  \end{proof}

In what follows, we extend the results from Section 3 on 2-party apportionment sequences and explain how to construct $n$-party sequences from all $\binom{n}{2}$ 2-party sequences.  To demonstrate the purely mechanical process, we first consider an example.  After the example, we explain why this process is valid in general.

\begin{example} \label{example-lift}
Let $p_1=16$, $p_2=11$, and $p_3 =7$ and $c \in [0,\frac{1}{9})$.  Applying Proposition \ref{constant_intervals}, we list the patterns of the periodic pieces of the pairwise apportionment sequences for $p_1, p_2$ and $c\in [0, \frac{1}{5})$ (first row in Table \ref{lift-example}); for $p_1, p_3$ and $c \in [0, \frac{1}{9} )$ (second row in Table \ref{lift-example}); and for $p_2, p_3$ and $c \in [0,\frac{1}{4} )$ (third row in Table \ref{lift-example}).  We will combine the 2-party sequences to form the periodic portion of a 3-party sequence of period $h = p_1 + p_2 + p_3=34$ for $c \in [0,\frac{1}{9})$, the intersection of $[0,\frac{1}{5})$, $[0,\frac{1}{9})$, and $[0,\frac{1}{4})$.

\begin{table}[htb]
\caption{The three rows in this table are the apportionment sequences for parties 1 and 2 for $c \in [0,\frac{1}{5})$, for parties 1 and 3 for $c\in [0,\frac{1}{9})$, and for parties 2 and 3 for $c \in [0, \frac{1}{4})$, respectively.}
\label{lift-example}
\renewcommand{\arraystretch}{1.2}
\setlength{\tabcolsep}{2pt}
\begin{tabular}{c*{26}c}
\textcolor{red}{$1$} & \textcolor{blue}{$2$} & \textcolor{red}{$1$} &  \textcolor{blue}{$2$} & \textcolor{red}{$1$} & \textcolor{blue}{$2$}  & \textcolor{red}{$1$} & \textcolor{red}{$1$} & \textcolor{blue}{$2$}  & \textcolor{red}{$1$} & \textcolor{blue}{$2$} &  \textcolor{red}{$1$} &\textcolor{red}{$1$} & \textcolor{blue}{$2$} &  \textcolor{red}{$1$} & \textcolor{blue}{$2$} & \textcolor{red}{$1$} & \textcolor{red}{$1$} & \textcolor{blue}{$2$} & \textcolor{red}{$1$} & \textcolor{blue}{$2$} & \textcolor{red}{$1$} & \textcolor{red}{$1$} & \textcolor{blue}{$2$} & \textcolor{red}{$1$}  & \textcolor{blue}{$2$} &  \textcolor{red}{$1$}  \\
\textcolor{red}{$1$} & $3$ & \textcolor{red}{$1$} & \textcolor{red}{$1$} & $3$ & \textcolor{red}{$1$} & \textcolor{red}{$1$} & $3$ & \textcolor{red}{$1$} & \textcolor{red}{$1$} & $3$ & \textcolor{red}{$1$} & \textcolor{red}{$1$} & \textcolor{red}{$1$} & $3$ & \textcolor{red}{$1$} & \textcolor{red}{$1$} & $3$ & \textcolor{red}{$1$} & \textcolor{red}{$1$} & $3$ & \textcolor{red}{$1$} & \textcolor{red}{$1$} \\
\textcolor{blue}{$2$} & $3$ & \textcolor{blue}{$2$} & $3$ & \textcolor{blue}{$2$} & \textcolor{blue}{$2$} & $3$ & \textcolor{blue}{$2$} & $3$ & \textcolor{blue}{$2$} & \textcolor{blue}{$2$} & $3$ & \textcolor{blue}{$2$} & $3$ & \textcolor{blue}{$2$} & \textcolor{blue}{$2$} & $3$ & \textcolor{blue}{$2$} \end{tabular} 
\end{table}

To construct the 3-party sequence, we examine the three 2-party sequences from left to right in Table \ref{lift-example}.  The top two sequences each begin with a 1.  Hence, the 3-party sequence begins with 1.  Now, imagine crossing out the 1s. The leading left (non-crossed out) entries in rows 1, 2, and 3 are now 2, 3, and 2, respectively.  Two of the rows match with 2, and 2 is the next term in the 3-party sequence.  To determine the third term in the 3-party sequence, we notice that the non-used leading left entries are now 1, 3, and 3, which implies that the third term in the sequence is 3.  Continuing in this fashion gives the 3-party apportionment sequence of period $34$:

\begin{center}
\renewcommand{\arraystretch}{1.2}
\setlength{\tabcolsep}{2pt}
\begin{tabular}{c*{34}c}
\textcolor{red}{$1$} & \textcolor{blue}{$2$} & 3 & \textcolor{red}{$1$} & \textcolor{blue}{$2$} &\textcolor{red}{$1$} & 3 & \textcolor{blue}{$2$} & \textcolor{red}{$1$} & \textcolor{red}{$1$} & \textcolor{blue}{$2$} & 3 & \textcolor{red}{$1$} & \textcolor{blue}{$2$} & \textcolor{red}{$1$} & 3 & \textcolor{red}{$1$} & \textcolor{blue}{$2$} & \textcolor{red}{$1$} & \textcolor{blue}{$2$} & \textcolor{red}{$1$} & 3 & \textcolor{red}{$1$} & \textcolor{blue}{$2$} & \textcolor{red}{$1$} & $3$ &\textcolor{blue}{$2$} & \textcolor{red}{$1$} & \textcolor{red}{$1$} & \textcolor{blue}{$2$} & $3$ & \textcolor{red}{$1$} & \textcolor{blue}{$2$} & \textcolor{red}{$1$}.  \end{tabular} \end{center}

Notice that rows 1, 2, and 3 of Table \ref{lift-example} are of length $27$, $23$, and $18$, because each pair of $p_i$ and $p_j$ are relatively prime and the sequences are periodic of length $p_i + p_j$.  It is not an accident that $\frac{27+23+18}{2} = 34$. This ensures that all numbers of the 2-party sequences were crossed out in the construction of the 3-party sequence. \end{example}

We refer to the  process described in Example  \ref{example-lift} as {\it lifting}. More generally, given  $n$ parties  with votes $p_i$ for $i =1$ to $n$ and cutpoint $c$, we can define an algorithm to lift the $\binom{n}{2}$ 2-party apportionment sequences  corresponding to $c$ to an $n$-party apportionment sequence for $c$.  The algorithm begins by generating the $\binom{n}{2}$ 2-party apportionment sequences.  Because of the tie-breaking rule and how $p_i \ge p_j$ for $i < j$, then the $n-1$ sequences that contain $1$ (and one of the other party's numbers between $2$ and $n$) all begin with a $1$.  Once this $1$ is used to begin the sequence, then the key  observation is that there is always a unique number between $1$ and $n$ that begins $n-1$ of the sequences. The party with this number is awarded the next seat in the $n$-party apportionment sequence.  After awarding this seat,  that number is removed from its leftmost position in the $n-1$ sequences and the process is repeated.  This process is formally defined in Algorithm \ref{algorithm-lift}. 

\begin{algorithm} 
\caption{Lifting from all $\binom{n}{2}$ 2-party sequences to the $n$-party sequence}\label{algorithm-lift}
\begin{algorithmic}
    \State \textbf{Input:} $n, c,$ and $p_i$ for $i = 1$ to $n$
    \State Let $m = \gcd(p_1, p_2, \dots, p_n)$
    \State Let $p = (p_1 + p_2 + \cdots + p_n)/m$
    \State \textbf{Output:} combinatorial word $w = \{ w_i \}_{i=1}^p$ where $w_i \in \{1, 2, \dots, n\}$

    \For{$i = 1$ to $n$}
        \State Let $j = i + 1$
        \While{$j \le n$}
            \State Let $\ell = (p_i + p_j)/m$
            \State Generate the 2-party sequence for parties $i$ and $j$ of length $\ell$ (using $i$ and $j$)
            \State Return line
            \State $j = j + 1$
        \EndWhile
    \EndFor

    \For{$i = 1$ to $p$}
        \For{$j = 1$ to $\frac{n(n+1)}{2}$}
            \For{$k = 1$ to $n$}
                \State $n_k = 0$
            \EndFor
            \If{the leftmost entry in the $j$th 2-party sequence is $k$}
                \State $n_k = n_k + 1$
            \EndIf
        \EndFor
        \For{$k = 1$ to $n$}
            \If{$n_k = n - 1$}
                \State $w_i = k$
            \EndIf
        \EndFor
        \State Delete the leftmost entry of the rows that begin with $w_i$
    \EndFor
\end{algorithmic}
\end{algorithm}

The following theorem proves that the algorithm generates the $n$-party apportionment sequence. 

\begin{theorem} \label{lifting} 
Algorithm \ref{algorithm-lift} generates the $n$-party apportionment sequence from all $\binom{n}{2}$ 2-party sequences.
\end{theorem}
\begin{proof}
First, we show that at every step of the algorithm there are $n-1$ leftmost entries that are the same.  Assume that $k$ seats have been allocated in the $n$-party apportionment sequence and that each party $i$ receives $a_i$ seats.  The party that receives seat $k+1$ is a party $\ell$ that satisfies $\argmax_{i\in N} \frac{p_i}{a_i+c}$.  If more than one party satisfies the condition, then we apply the tie-break rule so that the party with the smallest party number receives the seat. There are $n-1$ 2-party apportionment sequences that have $\ell$ as one of the two parties.  Because $\frac{p_\ell}{a_\ell+c} \ge \frac{p_i}{a_i+c}$ for all $i$, then $\ell$ appears as the leftmost entry in those $n-1$ sequences. Each number appears in exactly $n-1$ sequences, so no other number can appear as the leftmost entry in $n-1$ sequences. Hence, there is always a unique party that is leftmost in $n-1$ sequences, as required in Algorithm \ref{algorithm-lift}. 

Let $m = \gcd(p_1, p_2, \ldots, p_n)$.  Now we show that $\frac{p_1 + p_2 + \cdots + p_n}{m}$ seats are allocated and that no extra terms are left in the 2-party sequences.  
For parties $i$ and $j$, it is possible that $p_i$ and $p_j$ are not relatively prime and because of the relationship to the other $p_k$, it may be necessary to have more than $\frac{p_i+p_j}{\gcd(p_i,p_j)}$ terms in the finite sequence.  To use the algorithm, we require the first $\frac{p_i + p_j}{m}$ terms from the apportionment sequence for parties $i$ and $j$. It follows that there are a total of 
$$\sum_{i \ne j} \frac{p_i + p_j}{m} = (n-1) \cdot \frac{p_1 + p_2 + \cdots + p_n}{m} $$
terms among the 2-party apportionment sequences.  Since each step of the algorithm removes $n-1$ terms, the first $\frac{p_1 + p_2 + \cdots + p_n}{m}$ seats are awarded from the period $\frac{p_1 + p_2 + \cdots + p_n}{m}$ apportionment sequence.  \end{proof}

\begin{example} \label{16_11_7}To determine all possible sequential apportionments for $p_1 = 16$, $p_2 = 11$, and $p_3 = 7$, we can apply Theorem \ref{lifting} for all $c$ by comparing the pairwise sequential apportionments for $p_1, p_2$ (Table \ref{pairwise-seq-appt-example}), for $p_1, p_3$ (Table \ref{first-sequential-example} with the 2s replaced by 3s), and for $p_2, p_3$ (Table \ref{pairwise-seq-appt-example}).  The combined apportionments appear in Table \ref{16-11-7}; this table includes colored entries to highlight how the sequences change as the value of $c$ increases. \end{example}

\begin{table}
 \caption{The pairwise sequential apportionments for $p_1=16, p_2=11$ and $p_2=11,p_3=7$.}
\label{pairwise-seq-appt-example}
\renewcommand{\arraystretch}{1.2}
\setlength{\tabcolsep}{2pt}
\begin{tabular}{c|*{23}c}
$c$ &  \multicolumn{22}{c}{Sequence for $p_1 = 16$ and $p_2 = 11$} \\
\hline 
$\left[ 0, \frac{1}{5} \right)$ &  $1$ & $2$ & $1$ &  $2$ & $1$ & $2$  & \textcolor{red}{$1^2$} & $2$  & $1$ & $2$ &  $1^2$ & $2$ &  $1$ & $2$ & $1^2$ & $2$ & $1$ & $2$ & $1^2$ & $2$ & $1$  & $2$ &  $1$  \\
$\left[ \frac{1}{5}, \frac{2}{5} \right)$ & $1$ & $2$ & $1$ & $2$ & \textcolor{red}{$1^2$} & $2$ & $1$ & $2$ & $1$ & $2$ & \textcolor{blue}{$1^2$} & $2$ & $1$ & $2$ & $1^2$ & $2$ & $1$ & $2$ & $1^2$ & $2$ &$1$ &$2$ &$1$ \\
$\left[ \frac{2}{5}, \frac{3}{5} \right)$ & $1$ & $2$ & $1$ & $2$ & $1^2$ & $2$ & $1$ & $2$ & \textcolor{blue}{$1^2$} & $2$ & $1$ & $2$ & $1$ & $2$ & \textcolor{red}{$1^2$} & $2$ & $1$ & $2$ & $1^2$ & $2$ &$1$ &$2$ &$1$ \\
$\left[ \frac{3}{5}, \frac{4}{5} \right)$ & $1$ & $2$ & $1$ & $2$ & $1^2$ & $2$ & $1$ & $2$ & $1^2$ & $2$ & $1$ & $2$ & \textcolor{red}{$1^2$} & $2$ & $1$ & $2$ & $1$ & $2$ & \textcolor{blue}{$1^2$} & $2$ &$1$ &$2$ &$1$ \\
$\left[ \frac{4}{5}, \frac{5}{5} \right)$ & $1$ & $2$ & $1$ & $2$ & $1^2$ & $2$ & $1$ & $2$ & $1^2$ & $2$ & $1$ & $2$ & $1^2$ & $2$ & $1$ & $2$ & \textcolor{blue}{$1^2$} & $2$ & $1$ & $2$ &$1$ &$2$ &\textcolor{red}{$1$} \\
$1$ & $1$ & $2$ & $1$ & $2$ & $1^2$ & $2$ & $1$ & $2$ & $1^2$ & $2$ & $1$ & $2$ & $1^2$ & $2$ & $1$ & $2$ & $1^2$ & $2$ & $1$ & $2$ & \textcolor{red}{$1^2$} &$2$ &
 \end{tabular} 
\renewcommand{\arraystretch}{1.2}
\setlength{\tabcolsep}{2pt}
\begin{tabular}{c|*{15}c}
$c$ & \multicolumn{14}{c}{Sequence for $p_2 = 11$ and $p_3 = 7$}  \\
\hline 
$\left[ 0, \frac{1}{4} \right)$ & $2$ & $3$ & $2$ & $3$ & $2^2$ & $3$ & $2$ & $3$ & $2^2$ & $3$ & $2$ & $3$ & \textcolor{red}{$2^2$} & $3$ & $2$ \\
$\left[ \frac{1}{4}, \frac{2}{4} \right)$ & $2$ & $3$ & $2$ & $3$ & $2^2$ & $3$ & $2$ & $3$ & \textcolor{blue}{$2^2$} & $3$ & \textcolor{red}{$2^2$} & $3$ & $2$ & $3$ & $2$ \\
$\left[ \frac{2}{4}, \frac{3}{4} \right)$ & $2$ & $3$ & $2$ & $3$ & \textcolor{red}{$2^2$} & $3$ & \textcolor{blue}{$2^2$} & $3$ &$2$ & $3$ & $2^2$ & $3$ & $2$ & $3$ & $2$ \\
$\left[ \frac{3}{4}, \frac{4}{4} \right)$ & $2$ & $3$ & \textcolor{red}{$2^2$} & $3$ & $2$ & $3$ & $2^2$ & $3$ &$2$ & $3$ & $2^2$ & $3$ & $2$ & $3$ & \textcolor{blue}{$2$} \\
$1$ & $2$ & $3$ & $2^2$ & $3$ & $2$ & $3$ & $2^2$ & $3$ &$2$ & $3$ & $2^2$ & $3$ & \textcolor{blue}{$2^2$} & $3$ & 
 \end{tabular} 
\end{table}

\begin{table}
 \caption{The sequential apportionments for $p_1=16$, $p_2=11$, and $p_3 = 7$.}
 \label{16-11-7}
\renewcommand{\arraystretch}{1.2}
\setlength{\tabcolsep}{2pt}
\begin{tabular}{c|*{34}c}
$c$ & \multicolumn{34}{c}{Sequence for $p_1 = 16$, $p_2 = 11$, and $p_3 = 7$}  \\
\hline 
$[0, \frac{1}{9} )$ &1 & 2 & 3 & 1 & 2 & 1 & 3 & 2 & 1 & 1 & 2 & 3 & 1 & 2 & 1 & 3 & 1 & 2 & 1 & 2 & 1 & 3 & 1 & 2 & 1 & 3 & 2 & 1 & 1 & 2 & 3 & 1 & 2 & 1  \\
$[\frac{1}{9}, \frac{1}{5} )$ & 1 & 2 & 3 & 1 & 2 & 1 & 3 & 2 & 1 & 1 & 2 & 3 & 1 & 2 & 1 & \textcolor{red}{1} & \textcolor{red}{3} & 2 & 1 & 2 & 1 & 3 & 1 & 2 & 1 & 3 & 2 & 1 & 1 & 2 & 3 & 1 & 2 & 1 \\
$[\frac{1}{5}, \frac{2}{9} )$ &  1& 2&3&1&2&1&3&\textcolor{red}{1}&\textcolor{red}{2}&1&2&3&1&2&1&1&3&2&1&2&1&3&1&2&1&3&2&1&1&2&3&1&2&1\\
$[\frac{2}{9}, \frac{1}{4} )$ &1&2&3&1&2&1&3&1&2&1&2&3&1&2&1&1&3&2&1&2&1&3&1&2&1&3&2&1&1&2&\textcolor{red}{1}&\textcolor{red}{3}&2&1\\
$[\frac{1}{4}, \frac{3}{9} )$ &1&2&3&1&2&1&3&1&2&1&2&3&1&2&1&1&3&2&1&2&1&3&1&2&1&\textcolor{red}{2}&\textcolor{red}{3}&1&1&2&1&3&2&1\\
$[\frac{3}{9}, \frac{2}{5} )$ &1&2&3&1&2&1&3&1&2&1&2&\textcolor{red}{1}&\textcolor{red}{3}&2&1&1&3&2&1&2&1&3&1&2&1&2&3&1&1&2&1&3&2&1  \\
$[\frac{2}{5}, \frac{4}{9} )$ &1&2&3&1&2&1&3&1&2&1&2&1&3&\textcolor{red}{1}&\textcolor{red}{2}&1&3&2&1&2&1&3&1&2&1&2&3&1&1&2&1&3&2&1  \\
$[\frac{4}{9}, \frac{2}{4} )$ &1&2&3&1&2&1&3&1&2&1&2&1&3&1&2&1&3&2&1&2&1&3&1&2&1&2&\textcolor{red}{1}&\textcolor{red}{3}&1&2&1&3&2&1  \\
$[\frac{2}{4}, \frac{5}{9} )$ &1&2&3&1&2&1&3&1&2&1&2&1&3&1&2&1&\textcolor{red}{2}&\textcolor{red}{3}&1&2&1&3&1&2&1&2&1&3&1&2&1&3&2&1  \\ 
$[\frac{5}{9}, \frac{3}{5} )$ &1&2&3&1&2&1&\textcolor{red}{1}&\textcolor{red}{3}&2&1&2&1&3&1&2&1&2&3&1&2&1&3&1&2&1&2&1&3&1&2&1&3&2&1  \\ 
$[\frac{3}{5}, \frac{6}{9} )$ &1&2&3&1&2&1&1&3&2&1&2&1&3&1&2&1&2&3&1&\textcolor{red}{1}&\textcolor{red}{2}&3&1&2&1&2&1&3&1&2&1&3&2&1  \\ 
$[\frac{6}{9}, \frac{3}{4} )$ &1&2&3&1&2&1&1&3&2&1&2&1&3&1&2&1&2&3&1&1&2&\textcolor{red}{1}&\textcolor{red}{3}&2&1&2&1&3&1&2&1&3&2&1  \\ 
$[\frac{3}{4}, \frac{7}{9} )$ &1&2&3&1&2&1&1&\textcolor{red}{2}&\textcolor{red}{3}&1&2&1&3&1&2&1&2&3&1&1&2&1&3&2&1&2&1&3&1&2&1&3&2&1  \\ 
$[\frac{7}{9}, \frac{4}{5} )$ &1&2&\textcolor{red}{1}&\textcolor{red}{3}&2&1&1&2&3&1&2&1&3&1&2&1&2&3&1&1&2&1&3&2&1&2&1&3&1&2&1&3&2&1  \\ 
$[\frac{4}{5}, \frac{8}{9} )$ &1&2&1&3&2&1&1&2&3&1&2&1&3&1&2&1&2&3&1&1&2&1&3&2&1&\textcolor{red}{1}&\textcolor{red}{2}&3&1&2&1&3&2&1  \\ 
$[\frac{8}{9}, 1 )$ &1&2&1&3&2&1&1&2&3&1&2&1&3&1&2&1&2&\textcolor{red}{1}&\textcolor{red}{3}&1&2&1&3&2&1&1&2&3&1&2&1&3&2&1  \\ 
$1$ &1&2&1&3&2&1&1&2&3&1&2&1&3&1&2&1&2&1&3&1&2&1&3&2&1&1&2&3&1&2&1&\textcolor{red}{1}&\textcolor{red}{2}&\textcolor{red}{3}  \\ 
 \end{tabular}
 \end{table}
 
Consistency also allows us to work backwards: if given a sequence, we can apply previous 2-party results to determine if the sequence was generated by a stationary divisor method. Notice that we do not require any information about the number of votes parties receive and we are not given a possible cutpoint $c$.  The following theorem acts as a characterization of all apportionment sequences for $n$ parties generated by stationary divisor methods.  For simplicity, Theorem \ref{working-backwards} excludes the case in which $q_i = q_j$.  We address this case in the discussion before Corollary \ref{last-cor}.

\begin{theorem} \label{working-backwards}
For a periodic apportionment sequence of $n$ parties $S=s_1, s_2, s_3, \ldots$ of period $Q$, let $q_i$ be the number of times that $i$ appears in the first $Q$ terms. Assume $q_i > q_j$ for $i < j$. Let  $S_{i,j}$ be the subsequence created by extracting all the $i$'s and $j$'s, for all  $i<j$, and let $S^\prime_{i, j}$ correspond to $S_{i,j}$ with $i$ and $j$ replaced by $1$ and $2$, respectively.  

Then $S$ is generated by a stationary divisor method if and only if $S_{i,j}^\prime$ for all $i < j$ satisfies Equations (\ref{eq_k1}) and (\ref{eq_ki}) from Proposition \ref{prop:increasing} with $p_i = q_i / \gcd(q_i,q_j)$ and $p_j = q_j / \gcd(q_i,q_j)$.  If that is the case, then $S$ is generated by a stationary divisor method with cutpoint $c$ for all $c$ in that intersection of the intervals in Equation (\ref{c_interval}).  
\end{theorem}

\subsection{Counting the number of distinct apportionment sequences for $n$ parties}
 
In this subsection, we  determine the number of distinct apportionment sequences for $n$ parties. Recall  from Section \ref{2parties}, if $p_1$ and $p_2$ are relatively prime with $p_1>p_2$, the interval $[0, 1)$ can be partitioned into $p_1-p_2$ intervals of the form $[\frac{\ell}{p_1-p_2}, \frac{\ell}{p_1-p_2})$.  Let $A=\{0, \frac{1}{p_1-p_2},  \frac{2}{p_1-p_2}, \ldots,  \frac{p_1-p_2-1}{p_1-p_2}\}$ be the set of left endpoints of these intervals. Then the number of apportionment sequences is equal to $\vert A \vert + 1$, where the extra $1$ corresponds to $c=1$.  If $p_1$ and $p_2$ are not relatively prime, the same statement holds but with $p_1$ and $p_2$ replaced by $p_1/\gcd(p_1, p_2)$ and $p_2/\gcd(p_1, p_2)$, respectively. 

Now suppose there are $n$ parties where all the $p_i$ are distinct; we address the case where the $p_i$ are not distinct separately. Each of the $\binom{n}{2}$ pairs generates its own partition of the interval $[0, 1)$ and its own set of left endpoints.   When we combine these pairs to enumerate the sequences for all $n$ parties, as in Example \ref{16_11_7}, we obtain a refinement of each of the partitions of $[0, 1)$ based on the  union of  the sets of leftmost endpoints from all the pairs.

Let $M=\{1, 2, \ldots, \binom{n}{2}\}$ be the set of indices of all possible pairs of parties (the order is irrelevant). For each $k \in M$, if $k$ corresponds to $p_i$ and $p_j$ with $p_i>p_j$, then let $m_k = (p_i  - p_j)/\gcd(p_i, p_j)$ and let $A_k = \{0, 1/m_k, 2/m_k, \ldots, (m_k - 1)/m_k \}$.  It follows that the number of possible sequential apportionments under all stationary divisors  is 
$$\left\vert \bigcup_{k \in M} A_k \right\vert +1,$$
where the ``$+1$" comes from $c=1$.  To count the number of elements in the union, we apply the principle of inclusion/exclusion in Theorem \ref{counting}, which will keep track of double, triple, etc. counting, which happens when the $m_k$ are not distinct.  The following lemma gives conditions for an $m_k$ to appear in every set of a collection of $A_{\ell}$.  

\begin{lemma} \label{intersection} Let $\mathbf p = (p_1, \ldots, p_n)$ with $p_i \ne p_j$ for all $i$ and $j$.
Let $A_k=\{0, \frac{1}{m_k}, \ldots, \frac{m_k-1}{m_k}\}$ for some $k \in M$, where $m_k = \frac{p_r + p_s}{\gcd(p_r,p_s)}$ for some $r$ and $s$. Then, for all $1 \le i_1 < i_2 < \cdots <i_k  \le \binom{n}{2},$ 
$$A_{i_1} \cap  \cdots \cap A_{i_k}  =\left\{0, \frac{1}{m}, \ldots, \frac{m-1}{m}\right\},$$ 
where $m= \gcd(m_{i_1}, \ldots, m_{i_k}).$
\end{lemma}

To see that the lemma holds, consider the two cases: the $m_{i_j}$ share a common factor $m>1$ and the $m_{i_j}$ are relatively prime collectively.  If $m= \gcd(m_{i_1}, \ldots, m_{i_k}) > 1$, then the leftmost endpoint $\frac{1}{m}$ appears in each $A_{i_j}$, as well as every other $\frac{\ell}{m}$  for $\ell = 2$ to $m-1$. If the $m_{i_j}$ are relatively prime collectively, then $A_{i_1} \cap  \cdots \cap A_{i_k}  = \{0\}$, and the sets only share the leftmost endpoint $0$. The following theorem requires Lemma \ref{intersection} and is a direct application of the principle of inclusion/exclusion \cite[p.165]{Br}.

  \begin{theorem} \label{counting}
Let $p_1, p_2, \dots, p_n$ be the vote totals of $n$ parties where $p_1>p_2 > \cdots > p_n$.  The number of possible sequential apportionments from stationary divisor methods is
\begin{align*}
1+\left\vert \bigcup_{k=1}^{\binom{n}{2}}  A_k \right\vert  &=1+  \sum_{\{i_1, \ldots, i_k\} \subset M}  (-1)^{k+1} \vert A_{i_1} \cap \cdots \cap A_{i_k} \vert  \\
&=1+  \sum_{\{i_1, \ldots, i_k\} \subset M}  (-1)^{k+1}  \gcd(m_{i_1}, \ldots, m_{i_k}),
\end{align*}
where $gcd(m_{i_j})$ is defined to be $m_{i_j}$.
\end{theorem}

The following example demonstrates the application of Theorem \ref{counting}, as well as shows all intervals of cutpoints that give rise to distinct apportionment sequences.

\begin{example} Let $\mathbf p = (25,17,13,5)$.  Then, $M = \{1, 2, 3, 4, 5, 6=\binom{4}{2}$\}.  Define
\begin{alignat*}{3}
&m_1 = \textstyle\frac{25-17}{1}=8 \quad &&m_2 =\textstyle\frac{25-13}{1} =12 \quad &&m_3 = \textstyle\frac{25-5}{5}=4 \\
&m_4 =\textstyle\frac{17-13}{1}=4 \quad &&m_5 = \textstyle\frac{17-5}{1}=12 \quad &&m_6 = \textstyle\frac{13-5}{1}=8. \end{alignat*}
Because $m_1 = m_6$, $m_2 = m_5$, and $m_3 = m_4$, then $\cup_{k=1}^6 A_k =\cup_{k=1}^3 A_k$.  To give insight into the proof of Theorem \ref{counting}, we show how the principle of inclusion/exclusion works for $\cup_{k=1}^3 A_k$:
\begin{align*}
&\vert A_1 \cup A_2 \cup A_3 \vert  = \vert \{0, \textstyle\frac{1}{8}, \textstyle\frac{2}{8}, \ldots, \textstyle\frac{7}{8} \} \cup \{0, \textstyle\frac{1}{12}, \textstyle\frac{2}{12}, \ldots, \textstyle\frac{11}{12} \} \cup \{0, \textstyle\frac{1}{4}, \textstyle\frac{2}{4}, \textstyle\frac{3}{4} \} \vert  \\
&= \left( \vert A_1 \vert + \vert A_2 \vert + \vert A_3 \vert \right) - \left( \vert A_1 \cap A_2 \vert + \vert A_1 \cap A_3 \vert + \vert A_2 \cap A_3 \vert \right) - \vert A_1 \cap A_2 \cap A_3 \vert \\
&= \left( 8+12 +4 \right) - \left( 4 + 4 + 4\right) + 4 = 16.\end{align*}
This value is verifiable because we can count the number of elements in $\cup_{k=1}^3 A_k$ directly:
$$\vert \cup_{k=1}^3 A_k \vert = \vert \textstyle{ \left\{ 0, \frac{1}{12}, \frac{1}{8}, \frac{2}{12}, \frac{1}{4}, \frac{4}{12}, \frac{3}{8}, \frac{5}{12}, \frac{2}{4}, \frac{7}{12}, \frac{5}{8}, \frac{8}{12}, \frac{3}{4}, \frac{10}{12}, \frac{7}{8}, \frac{11}{12} \right\}} \vert = 16.$$ 

Alternatively, we can apply Theorem \ref{counting} to show $\left\vert \cup_{k=1}^6 A_k \right\vert = 48 - 72 + 80 - 60 + 24 - 4 = 16$, as expected.  Thus, there are $16+ 1 = 17$ possible sequential apportionments under stationary divisor methods, when the D'Hondt apportionment for $c=1$ is included. Each of the following intervals for $c$ give a distinct apportionment sequence:
\begin{align*}
&\textstyle{ \left[ 0, \frac{1}{12} \right), \left[ \frac{1}{12}, \frac{1}{8} \right), \left[ \frac{1}{8}, \frac{2}{12} \right), 
\left[ \frac{2}{12}, \frac{1}{4} \right), \left[ \frac{1}{4}, \frac{4}{12} \right),  \left[  \frac{4}{12}, \frac{3}{8} \right), \left[ \frac{3}{8}, \frac{5}{12} \right), \left[ \frac{5}{12}, \frac{2}{4} \right),}\\
&\textstyle{ \left[ \frac{2}{4}, \frac{7}{12} \right),
\left[ \frac{7}{12}, \frac{5}{8} \right), \left[ \frac{5}{8}, \frac{8}{12} \right] ,
\left[ \frac{8}{12}, \frac{3}{4} \right), \left[ \frac{3}{4}, \frac{10}{12} \right), \left[ \frac{10}{12}, \frac{7}{8} \right), \left[ \frac{7}{8}, \frac{11}{12} \right), \{ 1 \}.} 
\end{align*}  \end{example}

If the party votes are not distinct, Theorems  \ref{counting}  and \ref{working-backwards}, as well as other results, can easily be modified. Note  that if two parties have equal votes (say   $p_i = p_{i+1}$), then the addition of party $i+1$ does not introduce any new endpoints to the 2-party apportionment sequences with parties with non-equal vote totals.  Additionally, whenever one of the parties with equal vote totals is awarded a seat (the one with the least index by the tie-breaking rule), then the next largest indexed party with the same vote total gets the next seat, and this continues until all parties with equal vote totals get a seat. This leads to the following corollary.

\begin{corollary} \label{last-cor} Let $p_1, p_2, \dots, p_n$ be the vote totals of parties with $p_i \ge p_j$ for $i < j$. Let $N^*$ be a largest subset of $N$ with all $p_i$ distinct.  Relabel the elements in $N^*$ as $q_1, q_2, \ldots, q_m$ with $q_i > q_j$ for $i < j$.  Then, the number of possible $n$-party apportionment sequences of the parties in $N$ is the number of possible $m$-party apportionment sequences of the parties in $N^*$. \end{corollary}

\section{Conclusion}\label{endremarks}

The selection of the Northern Ireland Assembly Executive Committee was based on several pre-determined agreements as well as the sequential use of D'Hondt's method. We can use this example as a way to synthesize our results and see how the sequences used to award minister positions would have changed under different stationary methods. 

Referring back to Example \ref{NIreland},  the top four parties who received ministerial positions were SF (Sinn F\' ein), DUP (Democratic Unionist Party), AP (Alliance Party), and UUP (Ulster Unionist Party) with vote distribution $\mathbf p = (27, 25, 17, 9)$. Applying Theorem  \ref{counting} to this data,
 there are  25 possible different sequences arising from stationary apportionment methods. Of course, these have length $27+25+17+9= 78$. If we are interested only in the first 8 positions (because SF and DUP nominated the First and deputy First Ministers, respectively), we can compare sequences directly.  Under D'Hondt's method, the first 8 terms of the sequence are:   $1, 2, 3, 1, 2, 1, 4, 3$ (where we have used the tie break to award seat 6 to party 1).  Under Adams' method, applying Theorem \ref{jeff_adams},  the first 8 terms of the sequence are:   $1, 2, 3, 4, 1, 2, 3, 1$. Both of these methods begin $1, 2, 3$ and both award the four parties $3, 2, 2,$ and $1$ seats, respectively (as occurred in real life); hence, this will be true if any stationary divisor method is applied (and any nonstationary divisor method, too). However, the two sequences could have different political consequences if used to select executive committee positions one at a time.  The sequences for all other stationary divisor methods would be lexicographically between the sequences from the Adams and the D'Hondt methods.
 
 Note, because AP received the Justice Minister position and their first position in the sequence was skipped (which we denote by parentheses), the actual sequence used in Northern Ireland was $1, 2, (3), 1, 2, 4, 1, 3$.  This differs from above because their tie-break awarded UUP seat 6 instead of awarding it to SF.  It would be interesting to investigate the consequences of using the apportionment sequence to also award the First and deputy First Minister positions.

Awarding representative seats sequentially in a consistent manner is a consequence of an apportionment method satisfying house monotonicity. Although apportionment methods have been applied to award cabinet minister positions sequentially and to just-in-time sequencing problems, the study of what is possible under sequential apportionment is underdeveloped.  This work is a first in trying to understand both properties of $n$-apportionment sequences as well as the number of possible sequences generated by stationary divisor methods.  Because divisor methods are weakly proportional, every party will eventually receive a number of seats equal to its vote totals if the house size is the sum of the votes.  This motivates an alternate way of viewing size bias to include {\it when} a parties receives its representatives, as opposed to how many.  This has lead to a partial order on sequences bookended by the sequences from the methods of Adams and D'Hondt, which favor, respectively, the smallest and largest parties. 

For 2-party sequences, we were able to understand all possible sequences.  And, by using the 2-party sequences, we were able to determine how many possible $n$-party apportionment sequences there are.  There is still room for additional work in this area.  By understanding what sequences are possible, it may entice practitioners to develop additional applications of sequential apportionment besides the awarding of cabinet minister positions and its use in just-in-time sequencing.

 \section*{Appendix}
 
\noindent {\bf Proposition 5} {\it 
Let $p_1$ and $p_2$ be relatively prime with $p_1>p_2$.  Assume an apportionment sequence of length $p_1 + p_2$ has the form 
$$1^{k_1}, \, 2, \,  1^{k_2}, \, 2,  \, \ldots \, , 1^{k_j}, \, 2, \, \ldots , \, 1^{k_{p_2}},  \, 2, \, 1^{p_1-(k_1+\cdots +k_{p_2})}.$$ 
\begin{enumerate} 
\item The apportionment sequence is generated by a stationary divisor method if and only if
   \begin{equation*}\tag{1} 1 \le k_1 \le \left\lfloor \frac{p_1}{p_2} \right\rfloor \end{equation*} and
   \begin{equation*}\tag{2} \left\lfloor\frac{(i^\prime-i)p_1}{p_2}\right\rfloor \le k_{i+1}+ \cdots + k_{i^\prime} \le  \left\lfloor \frac{(i^\prime-i)p_1}{p_2}\right\rfloor + 1 \mbox{ for all } 1 \le i < i^\prime \le p_2. \end{equation*}

\item If these conditions from inequalities (\ref{eq_k1}) and (\ref{eq_ki}) are met, then the cutpoint $c$ lies in the interval 
\begin{equation*} \tag{3}
 \left[ \max_{i \ge 1} \frac{p_2(\sum_{j \le i} k_j-1) -p_1(i-1)}{p_1-p_2}, \,\, \min_{i \ge 1} \frac{p_2(\sum_{j \le i} k_j)-p_1(i-1)}{p_1-p_2}\right).\end{equation*}
Moreover, if  $\delta_c =  \frac{(p_1-p_2)}{p_2}c$, then 
\begin{equation*} \tag{4}
\begin{aligned}
 k_1 &=\lfloor  \delta_c \rfloor + 1, \\
 k_i & =\left\lfloor \frac{p_1}{p_2}(i-1) + \delta_c \right\rfloor - \left\lfloor \frac{p_1}{p_2}(i-2) + \delta_c \right\rfloor \quad \mbox{ for } i=2, \ldots, p_2, \text{ and}\\
k_{p_2+1} &= p_1 - \left\lfloor  \frac{p_1}{p_2}(p_2-1)+\delta_c \right\rfloor -1.\\
\end{aligned} 
\end{equation*}
\end{enumerate}}

\begin{proof} Proof of part {\it 1}. 
To  show that (\ref{eq_k1}) and (\ref{eq_ki}) are necessary,  suppose that the sequence is generated by a stationary divisor method with cutpoint $c$. From the proof of part {\it 2},  recall this implies
\begin{equation*} \tag{5}
\frac{p_1}{\sum_{j \le i}k_j -1+c} \ge \frac{p_2}{i-1+c}  > \frac{p_1}{\sum_{j \le i}k_j +c}
\end{equation*}
or
\begin{equation*}\tag{6}
\frac{p_2(\sum_{j \le i}k_j -1)-p_1(i-1)}{p_1-p_2} \le c < \frac{p_2(\sum_{j \le i}k_j) -p_1(i-1)}{p_1-p_2}.
\end{equation*}
Substituting $i=1$ into Equation (\ref{cratios}), we get
\begin{equation*}
\frac{p_2(k_1-1)}{p_1-p_2} \le c < \frac{p_2k_1}{p_1-p_2}.
\end{equation*}
Because $c \in [0,1]$, this implies  $\frac{p_2(k_1-1)}{p_1-p_2}<1$ and $\frac{p_2(k_1)}{p_1-p_2}>0$. Solving for $k_1$ and recalling that $k_1$ is an integer, it follows that  $1 \le k_1 \le \left\lfloor \frac{p_1}{p_2} \right\rfloor$.

To show that  Equation (\ref{eq_ki}) is necessary, note that in order for $c$ to exist, the intervals in Equation (\ref{cratios}) must have nonempty intersection. Thus, each left end-point must be to the left of each right end-point. So for each $i \ne i^\prime,$
$$\frac{p_2(\sum_{j \le i}k_j -1)-p_1(i-1)}{p_1-p_2} < \frac{p_2\sum_{j \le i^\prime}k_j -p_1(i^\prime-1)}{p_1-p_2}.$$
If $i < i^\prime$, this implies $p_1(i^\prime-i) < p_2(k_{i+1} + \cdots + k_{i^\prime}+1)$.  If $i^\prime < i$ then  $p_1(i-i^\prime) > p_2(k_{i^\prime+1} + \cdots + k_{i}-1)$. Combining these (and switching the labels $i$ and $i^{\prime}$ in the second inequality so that $i^\prime > i$), we get
$$p_2(k_{i+1}+ \cdots + k_{i^\prime}-1) <p_1(i^\prime-i) < p_2(k_{i+1}+ \cdots + k_{i^\prime}+1).$$
 or 
$$\frac{(i^\prime-i)p_1}{p_2}-1 < k_{i+1}+ \cdots + k_{i^\prime} < \frac{(i^\prime-i)p_1}{p_2}+1.$$
Equation (\ref{eq_ki}) follows because the $k_i$ are integers.

Conversely, suppose that the apportionment sequence has the given form and that Equations (\ref{eq_k1}) and (\ref{eq_ki}) are satisfied. Then the interval in Equation (\ref{c_interval}) is non-empty and if $c$ lies in the interval, it clearly satisfies Equation (\ref{kratios}) for each $i$. Hence the stationary divisor method with cutpoint $c$ yields the same sequence. It remains to show that the interval (\ref{c_interval}) lies in $[0,1].$ To see this, note that for each $i \ge 1$,
\begin{align*}
p_2(k_1+ \ldots +k_i) - (i-1)p_1& \ge p_2\left(k_1+ \left\lfloor (i-1)\frac{p_1}{p_2} \right\rfloor \right)-(i-1)p_1\\
& =p_2\left(k_1+\left\lfloor (i-1)\frac{p_1}{p_2} \right\rfloor -(i-1)\frac{p_1}{p_2}\right) \\
 &\ge p_2 \left(1+\left\lfloor (i-1)\frac{p_1}{p_2} \right\rfloor -(i-1)\frac{p_1}{p_2}\right) >0.\\
\end{align*}
Thus, the right-hand side of the interval from Equation (\ref{c_interval}) is greater than 0. Similarly, we can show the left-hand side of the interval (\ref{c_interval}) is less than or equal to 1, implying that $c \in [0,1]$. 
\end{proof}

\noindent {\bf Theorem 3} {\it 
 Suppose  $p_1$ and $p_2$ are relatively prime and that $p_1=ap_2+b$ for some integers $a \ge 1$ and $0< b <p_2$. If $\delta_c =  \frac{(p_1-p_2)}{p_2} c= (a-1+\frac{b}{p_2})c$, then
 \begin{align*}\label{ki_a}
k_1 & = \lfloor \delta_c \rfloor + 1,\\
k_i  & =\left\{ \begin{array}{cl} a+1 & \mbox{ if } i \in S   \\  
a & \mbox{ if }  i \notin S  \end{array} \right. , \mbox{ and}\\
k_{p_2+1} &=a+b-1 - \left\lfloor \frac{b}{p_2} (p_2-1)+ \delta_c \right\rfloor
\end{align*}
where  $$S = \left\{i = 2, \ldots, p_2 : i = \left\lceil (j- \delta_c+ \left\lfloor  \delta_c \right\rfloor)\frac{p_2}{b} \right\rceil + 1 \mbox{, for  } j = 1, 2, \ldots, b\right\}.$$ }

\begin{proof} 
The formulas for $k_1$ and $k_{p_2+1}$ follow directly by substituting $p_1=ap_2+b$  into  Equation (\ref{ki_formulas}).   
 Doing the same for the remaining $k_i$, we obtain for $i=2, \ldots, p_2$,
  \begin{align*}k_i & =\left\lfloor \frac{p_1}{p_2}(i-1)+ \delta_c \right\rfloor - \left\lfloor  \frac{p_1}{p_2}(i-2) +\delta_c \right\rfloor \\
& = \left\lfloor a(i-1)+ \frac{b}{p_2}(i-1)+  \delta_c \right\rfloor - \left\lfloor a(i-2) +\frac{b}{p_2}(i-2) +  \delta_c\right\rfloor \\
& = a +  \left\lfloor  \frac{b}{p_2}(i-1)+   \delta_c \right\rfloor - \left\lfloor \frac{b}{p_2}(i-2) +\delta_c \right\rfloor.
\end{align*}
Since $\frac{b}{p_2}<1$, these exponents   are equal to either $a$ or $a+1$.  We claim  at most $b$ of the $k_i$  are equal to $a+1$. To see this, we substitute $i^\prime=p_2$, $i=1$ and $p_1=ap_2+b$ into Equation (\ref{eq_ki}), to obtain
$$\left\lfloor \frac{(p_2-1)ap_2+b}{p_2}\right\rfloor \le k_{2}+ \cdots + k_{p_2} \le  \left\lfloor \frac{(p_2-1)ap_2+b}{p_2}\right\rfloor + 1,$$ 
which implies $(p_2-1)a+ \left\lfloor\frac{(p_2-2)b}{p_2}\right\rfloor \le  k_{2}+ \cdots + k_{p_2} \le  (p_2-1)a+ \left\lfloor\frac{(p_2-2)b}{p_2}\right\rfloor+1.$ Since $(p_2-2)b/p_2<b$, the claim follows.

 Finally, to determine which of the $k_i$ is equal to $a+1$, consider the quantity  $K_{i} =\lfloor (i-1) \frac{b}{p_2} +\delta_c\rfloor$. If $i=1$, $K_i= \lfloor \delta_c \rfloor$. As $i$ increases, the values of $K_i$ are constant until some smallest $i$  (call it $i_1$) satisfies $K_{i_1} = \lfloor \delta_c \rfloor+1$; thus $i_1$ is the smallest  index such that $k_{i}=a+1.$ Similarly, if $i_2$ is the smallest $i$ such that $K_{i_2} = \lfloor \delta_c \rfloor+2$ then $i_2$ is the second smallest index  such that $k_{i}=a+1.$ More generally, the indices for which $k_i=a+1$ occur at the smallest indices $i_j$ for which  $K_{i_j} = \lfloor \delta_c \rfloor+j.$   

Let $\ell_j= \lceil (j-\delta_c+ \lfloor \delta_c \rfloor)\frac{p_2}{b} \rceil + 1$; we claim $i_j=\ell_j$  for $j \ge 1$. To see this, note that  
$$(j-\delta_c+ \lfloor \delta_c \rfloor )\frac{p_2}{b} \le   \ell_j-1< (j-\delta_c+ \lfloor \delta_c \rfloor )\frac{p_2}{b}+1,$$ 
which implies
$$j+ \lfloor \delta_c \rfloor \le   (\ell_j-1)\frac{b}{p_2} +\delta_c < j+ \lfloor \delta_c \rfloor +\frac{b}{p_2},$$ 
and hence $ \lfloor  (\ell_j-1)\frac{b}{p_2} +\delta_c  \rfloor = j+ \lfloor \delta_c \rfloor.$

On the other hand
$$(j-\delta_c+ \lfloor \delta_c \rfloor )\frac{p_2}{b}-1 \le   \ell_j-2 <(j-\delta_c+ \lfloor \delta_c \rfloor )\frac{p_2}{b},$$ 
which implies
$$j+ \lfloor \delta_c \rfloor -\frac{b}{p_2} \le   (\ell_j-2)\frac{b}{p_2} +\delta_c  < j+ \lfloor \delta_c \rfloor,$$ 
and so $ \lfloor  (\ell_j-2)\frac{b}{p_2} +\delta_c  \rfloor < j+ \lfloor \delta_c \rfloor.$
Thus, $\ell_j$ is the smallest integer such that  $K_{\ell_j} =\lfloor \delta_c \rfloor+j$ or $i_j=\ell_j$, as required. 
\end{proof}

\noindent {\bf Theorem 4} {\it 
Suppose $p_1$ and $p_2$  are relatively prime and $p_1>p_2+1$. Then the apportionment sequences  corresponding to stationary divisor methods are constant for  cutpoints in  the intervals  $c \in  [\frac{\ell}{p_1-p_2}, \frac{\ell+1}{p_1-p_2})$ where $\ell= 0, \ldots, p_{1}-p_2-1$.  Moreover, each of these $p_1-p_2$ sequences is distinct.}

\begin{proof}
 Suppose  $p_1=ap_2+b$ for some integers $a \ge 1$ and $0< b <p_2$ and suppose that $c \in [\frac{\ell}{p_1-p_2}, \frac{\ell+1}{p_1-p_2})$ for some $\ell = 0, \ldots, p_{1}-p_2-1$. Let $\ell = \nu p_2+\ell^\prime$ for some integers $\nu \ge 0$ and $0 \le \ell^\prime< p_2.$
Then $\delta_c = \frac{p_1-p_2}{p_2}c$ satisfies 
$\frac{\ell}{p_2} \le \delta_c <  \frac{\ell+1}{p_2}$, which implies  $\lfloor \delta_c \rfloor = \nu$ and  $k_1= \nu +1$. Thus, $k_1$ is constant on this interval. 

To show that the other $k_i$ are  constant on  $[\frac{\ell}{p_1-p_2}, \frac{\ell+1}{p_1-p_2})$ it suffices to show that the indices $i_j =  \lceil (j-\delta_c+ \lfloor \delta_c \rfloor)\frac{p_2}{b} \rceil + 1$, from the definition of $S$ in Theorem \ref{amax}, are also constant on this interval. But  
$$\frac{-\ell^\prime-1}{p_2}   < -\delta_c+\lfloor \delta_c \rfloor \le \frac{-\ell^\prime}{p_2},$$
which implies $$\frac{jp_2-\ell^\prime-1}{b} <  (j-\delta_c+\lfloor \delta_c \rfloor)\frac{p_2}{b} \le   \frac{jp_2-\ell^\prime}{b}.$$
So   $i_j =  \lceil (j+\lfloor \delta_c \rfloor -\delta_c)\frac{p_2}{b} \rceil+1 = \lceil  \frac{jp_2-\ell^\prime}{b} \rceil+1$ is constant on the interval as claimed.

Finally, to show that each interval of the form $c \in [\frac{\ell}{p_1-p_2}, \frac{\ell+1}{p_1-p_2})$ corresponds to a different apportionment sequence, it suffices to consider two adjacent intervals, since the sequences increase in lexicographical ordering as $c$ increases. 

Suppose that   $c \in  [\frac{\ell}{p_1-p_2}, \frac{\ell+1}{p_1-p_2})$ and $c^\prime \in  [\frac{\ell+1}{p_1-p_2}, \frac{\ell+2}{p_1-p_2})$ and that the two apportionment sequences are the same. Then
$\delta_{c^\prime} = \frac{p_1-p_2}{p_2}c^\prime$ satisfies $ \frac{\ell+1}{p_2} \le \delta_{c^\prime} <  \frac{\ell+2}{p_2}$, and recalling  that $\ell = \nu p_2+\ell^\prime,$ we obtain 
$\nu+\frac{\ell^\prime+1}{p_2}  \le \delta_{c^\prime} < \nu+\frac{\ell^\prime+2}{p_2}.$
Since the $k_1$ values for each sequence are equal, $\lfloor \delta_{c^\prime} \rfloor = \lfloor \delta_{c} \rfloor   = \nu$. Hence $\ell^\prime+1 < p_2$.

In addition, since the $k_i$ values for each sequence are equal, we must have for each $j$,
$$ \left\lceil  \frac{jp_2-\ell^\prime}{b} \right\rceil= \left\lceil (j-\delta_c+\lfloor \delta_c \rfloor )\frac{p_2}{b} \right\rceil  = \left\lceil (j-\delta_{c^\prime} +\lfloor \delta_{c^\prime} \rfloor)\frac{p_2}{b} \right\rceil =  \left\lceil  \frac{j p_2-(\ell^\prime+1)}{b} \right\rceil.$$

But $p_2$ is relatively prime to $b$. So there exists a $j^\prime\in \{1, \ldots, b\}$ such that $j^\prime p_2$ is congruent to $\ell^\prime + 1$ mod $b$. Hence  $\frac{j^\prime p_2-(\ell^\prime+1)}{b}$ is an integer, which implies $\lceil \frac{j^\prime p_2-\ell^\prime}{p_2} \rceil =  \lceil  \frac{j^\prime p_2-(\ell^\prime+1)}{p_2} \rceil+1.$ This contradiction shows that 
 the sequential apportionments corresponding to $c$ and $c^\prime$ are distinct.   \end{proof} 

\end{document}